\documentclass{amsart}
\usepackage{url}
\usepackage{amssymb}
\usepackage{graphicx}
\newtheorem{theorem}{Theorem}[section]
\newtheorem{lemma}[theorem]{Lemma}
\newtheorem{corollary}[theorem]{Corollary}

\newtheorem{definition}[theorem]{Definition}

\newtheorem{function}[theorem]{Function}
\newtheorem{algorithm}[theorem]{Algorithm}

\theoremstyle{definition}
\newtheorem{example}{Example}
\newtheorem{remark}{Remark}

\newcommand{\Z}{\mathbf{Z}}
\newcommand{\ZZ}{\mathbf{Z}}
\newcommand{\Q}{\mathbf{Q}}
\newcommand{\QQ}{\mathbf{Q}}
\newcommand{\OO}{\mathcal{O}}
\newcommand{\tors}{\operatorname{tors}}
\newcommand{\disc}{\operatorname{disc}}
\newcommand{\ra}{\rightarrow}
\begin{document}
\date{\today}
\title{Computation on Elliptic Curves with Complex Multiplication}
\author[P. Clark  \and P. Corn \and A. Rice \and J. Stankewicz 
]{Pete L. Clark \and Patrick Corn \and Alex Rice \and James Stankewicz 
}

\begin{abstract}

We give the complete list of possible torsion subgroups of elliptic curves with complex multiplication over number fields of degree 1-13. Additionally we describe the algorithm used to compute these torsion subgroups and its implementation.

\end{abstract}
\maketitle
\section{Introduction}

\subsection{The main results}
The goal of this paper is to present a complete list of possible torsion subgroups of elliptic curves with complex multiplication over number fields of small degree. Our main tool is an algorithm whose input is a positive integer $d$. The output is a (necessarily finite) list of isomorphism classes of finite abelian groups $G$ such that $G$ is isomorphic to $E(K)[\tors]$ for some number field $K$ of degree $d$ and some elliptic curve $E$ defined over $K$ with complex multiplication. 

Our algorithm requires a complete list of imaginary quadratic fields of class number $h$ for all integers $h$ which properly divide $d$. Fortunately, M. Watkins \cite{Watkins} has enumerated all imaginary quadratic fields with 
class number $h\le 100$, which would in theory allow us to run our algorithm for all $d \le 201$ 
(and for infinitely many other values of $d$, for instance all prime values). 

We implemented our algorithm using the MAGMA programming language and ran it on Unix 
servers in the University of Georgia Department of Mathematics. The result, after doing some additional analysis, is a complete list of torsion subgroups for degree $d$ with $1\le d \le 13$. This list, for each degree $d$, is described in 
Section~\ref{lists}.$d$.

For $d=1$ these computations were first done by L. Olson in 1974 \cite{Olson}, 
whereas for $d=2$ and $3$ they are a special case of work of H. Zimmer and his collaborators over a ten year period 
from the late 1980's to the late 1990's \cite{Zimmer1}, \cite{Zimmer2}, \cite{Zimmer3}. We believe that our results are new for $4 \le d \le 13$.

This work was begun during a VIGRE research group led by Pete L. Clark and Patrick Corn and attended by Brian Cook, Steve Lane, Alex Rice, James Stankewicz, Nathan Walters, Stephen Winburn and Ben Wyser at the University of Georgia. Alex Rice, James Stankewicz, Nathan Walters and Ben Wyser were partially supported by NSF VIGRE grant DMS-0738586 during this work. James Stankewicz was also partially supported by the Van Vleck fund at Wesleyan University. Special thanks go to Jon Carlson, who offered use of his MAGMA server and invaluable support with coding in MAGMA. Thanks to Bianca Viray for a helpful discussion of the proof of Lemma \ref{SPECIALIZATION}. Thanks also go to Andrew Sutherland, whose interest in this project demanded that this paper be polished into publishable form.

\subsection{Connections to prior work} According to the celebrated \textbf{uniform boundedness theorem} of L. Merel \cite{Merel}, 
for any fixed $d \in \Z_{> 0}$, the supremum of the size of all rational torsion 
subgroups of all elliptic curves defined over all number fields of degree $d$ is finite.  

In 1977, B. Mazur proved uniform boundedness for $d = 1$ (i.e., for elliptic curves $E_{/\Q}$) \cite{Mazur}.   
Moreover, Mazur gave a complete classification of the possible torsion subgroups:

\[
E(\QQ)[\tors]\in \begin{cases}\ZZ/m\ZZ & \textrm{for } m = 1,\dots,10,12, \\ \ZZ/2\ZZ \oplus\ZZ/2m\ZZ & \text{for $m = 1, \ldots, 4$.} \end{cases}
\]

Work of Kamienny \cite{Kamienny86}, \cite{Kamienny92} and of Kenku and Momose \cite{KenkuMomose}
gives the following result when $K$ is a quadratic number field:

\[
E(K)[\tors] \in \begin{cases} \ZZ/m\ZZ & \textrm{for } m=1,\dots,16,18, \\ \ZZ/2\ZZ \oplus \ZZ/2m\ZZ & \textrm{for } m = 1,\dots,6, \\ \ZZ/3\ZZ \oplus \ZZ/m\ZZ & \text{for $m=3,6,$} \\ and & \ZZ/4\ZZ \oplus \ZZ/4\ZZ. \end{cases}
\]

This and similar subsequent enumeration results over varying number fields are to be understood in the following sense.  First, for 
any quadratic field $K$ and any elliptic curve $E_{/K}$, the torsion subgroup of $E(K)$ is 
isomorphic to one of the groups listed.  Second, for each of the groups $G$ listed, there exists at 
least one quadratic field $K$ and an elliptic curve $E_{/K}$ with $E(K)[\tors] \cong G$. A complete classification of torsion subgroups of elliptic curves over cubic fields is not yet known. 

Further results come from focusing on particular classes of elliptic curves. Notably H. Zimmer and his collaborators 
have done extensive computations on torsion in elliptic curves with $j$-invariant in the ring of algebraic integers. In \cite{Zimmer1}, M\"uller, Stroher and Zimmer proved that in the case of integral $j$-invariant, if $K$ is a quadratic number field then 
\[
E(K)[\tors] \in \begin{cases} \ZZ/m\ZZ &\textrm{for } m=1,\dots,8,10, \\ \ZZ/2\ZZ \oplus \ZZ/m\ZZ &\textrm{for }m=2,4,6, \\ \textrm{and} &\ZZ/3\ZZ \oplus \ZZ/3\ZZ.\end{cases}
\]

In \cite{Zimmer3} Peth\"o, Weis and Zimmer showed that if $E$ has integral $j$-invariant and $K$ is a cubic number field then 
\[
E(K)[\tors] \in \begin{cases} \ZZ/m\ZZ &\textrm{for } m=1,\dots,10,14, \\ \ZZ/2\ZZ \oplus \ZZ/m\ZZ & \textrm{for } m=2,4,6.\end{cases}
\]

Here we study elliptic curves with complex multiplication.  Such curves form a subclass of 
curves with integral $j$-invariant \cite[Theorem. II.6.4]{SilvermanII}, so our results are subsumed by the above results for $d \le 3$; but, as we will see, the CM hypothesis allows us to extend our computations to higher values of $d$, up to $d=13$.

\section{Background}\label{background}

\subsection{Kubert-Tate normal form} The fundamental result on which our algorithm rests is the following elementary theorem, which gives a parameterization of all elliptic curves with an $N$-torsion point for $N \ge 4$.

\begin{theorem} \label{kubert} (Kubert) Let $E$ be an elliptic curve over a field $K$ and $P \in E(K)$ a point of order at least $4$. Then $E$ has an equation of the form
\begin{equation} \label{knf}
y^2 + (1-c)xy - by = x^3-bx^2
\end{equation}
for some $b,c \in K$, and $P = (0,0)$. \end{theorem}

\begin{proof}  This first appeared in \cite{Kubert}.  See for instance \cite{Zimmer1}, \S3.\end{proof}

We will call the equation (\ref{knf}) the {\em Kubert-Tate normal form} 
of $E$, and our notation for a curve in Kubert-Tate normal form with parameters $b, c$ as above will be simply $E(b,c)$. The $j$-invariant of this elliptic curve is
\begin{equation}\label{jinv}
j(b,c) = \frac{(16b^2 + 8b(1-c)(c+2) + (1-c)^4)^3}{b^3(16b^2 - b(8c^2+20c-1) - c(1-c)^3)}.
\end{equation}

\begin{remark} This form is unique for a given curve with a fixed point of order at least 4. In practice we use this to find elliptic curves with some primitive $N$-torsion point, so an elliptic curve $E$ may have many isomorphic Kubert-Tate normal forms, depending on which torsion point we choose to send to $(0,0)$. \end{remark}
 
\begin{example}
\label{Example1}
Here are some small multiples of the point $(0,0)$ on $E(b,c)$: 
\[ [2](0,0) = (b,bc), \]
\[ [3](0,0) = (c,b-c), \]
\[
[4](0,0) = \left( \frac{b(b-c)}{c^2}, \frac{b^2(c^2+c-b)}{c^3} \right),
\]
\[
[5](0,0) = \left( \frac{bc(c^2+c-b)}{(b-c)^2}, \frac{bc^2(b^2-bc-c^3)}{(b-c)^3} \right),
\]
\[ [6](0,0) = \left(\frac{(b-c)(b^2 - bc - c^3)}{A^2}, \frac{c(b-c)^2(2b^2- b c(c-3) + c^{2})}{A^3} \right), \]

\[
[7](0,0) = \left( \frac{Abc((b-c)^2 + Ab)}{(b^2-bc-c^3)^2}, \frac{(Ab)^2((b-c)^3 + c^3A)}{(b^2-bc-c^3)^3} \right),
\]
where $A = b-c-c^2$.  In particular we see that for $N \leq 3$, $(0,0)$ cannot be an $N$-torsion point on $E(b,c)$.  
\end{example}

\subsection{Modular curves} The affine modular curve $Y_1(N)$ for $N \ge 4$ is a fine moduli space for pairs $(E,P)$ where $E$ is an elliptic curve and $P$ is a point of exact order $N$ on $E$. We will search for CM-points on $Y_1(N)$ for various values of $N\ge 4$; that is, points over various number fields which correspond to CM elliptic curves with an $N$-torsion point(the $Y_1(N)$ for $1\le N\le 3$ are coarse moduli spaces and so will only give us the information we desire over an algebraically closed field). Kubert curves give a down-to-earth way of constructing a defining equation for $Y_1(N)$.

\begin{definition} Let $\QQ(b,c)$ be a rational function field, and let $E_{/\QQ(b,c)}$ denote the elliptic curve given by equation (\ref{knf}). If $N\ge 3$ is an integer, let $n_1,d_1,n_2,d_2\in \QQ[b,c]$ be such that $(n_i,d_i) =1$, $d_i$ is monic, and $$x\left(\left[\left\lceil\dfrac{N}{2}\right\rceil -1\right](0,0)\right) = \dfrac{n_1(b,c)}{d_1(b,c)}, x\left(\left[\left\lfloor\dfrac{N}{2} \right\rfloor+1\right](0,0)\right) = \dfrac{n_2(b,c)}{d_2(b,c)}.$$

Then we let $f_N(b,c) = n_1d_2 - n_2d_1 \in \QQ[b,c]$.
\end{definition}

\begin{lemma}
\label{SPECIALIZATION}
Let $k$ be a field, let $Y_{/k}$ be an integral algebraic variety, let $q: A \rightarrow Y$ be a relative 
abelian variety, and let $y$ be a closed point of $Y$.  Then the specialization 
map $\mathfrak{s}: A(K(Y)) \rightarrow A_y(k(y))$ is a group homomorphism.
\end{lemma}
\begin{proof}
This result appears in \cite[p. 40]{Lang}.  Lang's (wonderful) text is
rather informally written: many results, including this one, are given there without proof or reference.  For the convenience of the reader we give a proof.  \\
Step 1: Suppose $Y$ is a nonsingular curve.  Then $A_{/Y}$ is equal to the \textbf{N\'eron model} of its generic fiber, so the map $\mathfrak{s}$ is a homomorphism by \cite[Proposition I.2.8]{BLR}.  \\
Step 2: Suppose $Y$ is a singular curve.  Let $\pi: \tilde{Y} \ra Y$ be its normalization, and let $\tilde{y}$ be a closed point of $\tilde{Y}$ with $\pi(\tilde{y}) = y$. Let $\tilde{A} = \pi^*(q) \ra \tilde{Y}$ be the pullback of the family to $\tilde{Y}$.  Then the fiber of $\tilde{A}$ over $\tilde{y}$ is canonically identified with the fiber of $A$ over $y$, and thus the specialization map $\tilde{s}: \tilde{A}(K(\tilde{Y})) \ra A_{\tilde{y}}(k(\tilde{y}))$ is canonically identified with $\mathfrak{s}$.  We 
have reduced to Step 1.  \\
Step 3: In the general case we choose a chain of closed irreducible subvarieties 
$Y_0 = \{y\} \subset Y_1 \subset \ldots Y_d = Y$ containing $y$, with $\dim Y_i = i$.  We apply Step 2 repeatedly, specializing from the generic point of $Y_i$ 
to the generic point of $Y_{i-1}$.  
\end{proof}

\begin{lemma}
\label{PICKYLEMMA}
If $b_0,c_0 \in \overline \QQ$ and $E(b_0,c_0)$ is an elliptic curve given by equation (\ref{knf}) then the point $(0,0)$ on $E(b_0,c_0)$ is an $N$-torsion point if and only if $f_N(b_0,c_0) = 0$.\end{lemma}

\begin{proof} By Example \ref{Example1} we must have $N \geq 4$.  \\
Step 1: Suppose $[N](0,0) = O$.  We claim that $[\lceil\frac{N}{2}\rceil-1](0,0)$ and $[\lfloor \frac{N}{2} \rfloor + 1](0,0)$ are finite and have equal $x$-coordinates.  Indeed, if $[\lceil \frac{N}{2} \rceil - 1](0,0) = O$, then $(0,0)$ is $(2-2\lceil \frac{N}{2} \rceil +N)$-torsion, i.e., 
$2$-torsion if $N$ is even and $1$-torsion if $N$ is odd, contradicting Example 1.  A similar argument shows that $[\lfloor \frac{N}{2} \rfloor + 1](0,0)$ is finite.  Moreover, we have 
\[ \left[\left\lceil \frac{N}{2} \right\rceil -1\right](0,0) + \left[\left\lfloor \frac{N}{2} \right\rfloor + 1\right](0,0) = \left[\left\lceil \frac{N}{2} \right\rceil + \left\lfloor \frac{N}{2} \right\rfloor\right](0,0) = 
[N](0,0) = O, \]
so $[\lceil\frac{N}{2}\rceil-1](0,0)) = -[\lfloor \frac{N}{2} \rfloor + 1](0,0)$.  As for any points $P,Q$ on a Weierstrass elliptic curve, 
 $P = \pm Q$ if and only if $x(P) = x(Q)$, this establishes the claim.  Now, 
since $x([\lceil \frac{N}{2} \rceil - 1](0,0)) = x( [\lfloor \frac{N}{2} \rfloor + 1 ](0,0)) \in \overline \QQ$, if $d_1(b_0,c_0) = 0$ then also $n_1(b_0,c_0) = 0$ 
hence $f_N(b_0,c_0) = n_1(b_0,c_0)d_2(b_0,c_0) - n_2(b_0,c_0)d_2(b_0,c_0) = 0$.  Similarly if 
$d_2(b_0,c_0) = 0$.  Finally, if $d_1(b_0,c_0) d_2(b_0,c_0) \neq 0$, then 
\[ x(\left[\left\lceil \frac{N}{2} \right\rceil - 1\right](0,0)) = \frac{n_1(b_0,c_0)}{d_1(b_0,c_0)} = 
\frac{n_2(b_0,c_0)}{d_2(b_0,c_0)} = x( \left[\left\lfloor \frac{N}{2} \right\rfloor + 1 \right](0,0)), \]
so $f_N(b_0,c_0) = 0$.  \\
Step 2: Suppose that $f_N(b_0,c_0) = 0$.  Thus there must be at least one irreducible factor $g(b,c)$ of $f_N(b,c)$ such that $g(b_0,c_0) = 0$.  Then 
$Z = \Q[b,c]/\langle g(b,c) \rangle$ is an irreducible curve, and let $Y$ be obtained from $Z$ by removing the finite set of closed points on which the Kubert curve $E(b_0,c_0)$ becomes singular.  Then $E(b,c)$ 
gives a relative elliptic curve over $Y$, and the elliptic curve $E(b_0,c_0)$ 
is its specialization at the closed point $(b_0,c_0)$.  By construction, $[N](0,0) = O$ on the generic fiber of $Y$, so by Lemma \ref{SPECIALIZATION}, 
$[N](0,0) = O$ on $E(b_0,c_0)$.  
\end{proof}

\begin{example}\label{Example2} We will make use of the explicit formulas of Example \ref{Example1}.  \\
a) Let $N = 4$.  Setting $x((0,0)) = x([3](0,0))$ gives $f_4(b,c) = c$.  \\
b) Let $N = 5$.  The condition that $(0,0)$ is a $5$-torsion point is $b-c = 0$.  
Setting $x(([2](0,0)) = x([3](0,0))$ gives $f_5(b,c) = b-c$. \\
c) Let $N = 7$.  The condition that $(0,0)$ is a $7$-torsion point is  $b^2-bc-c^3 = 0$.  Setting $x([3](0,0)) = x([4](0,0))$ gives $f_7(b,c) = b(b-c) - c(c^2)$.
\end{example}





As Examples \ref{Example1} and \ref{Example2} illustrate, the complexity of the rational functions giving the coordinates of $[N](0,0)$ increases rapidly with $N$.  Our trick of computing $[\lceil \frac{N}{2} \rfloor -1](0,0)$ and $[\lfloor \frac{N}{2} \rfloor + 1](0,0)$ instead becomes a critical one to extend the range of 
our calculations.  

In general, $f_N(b,c) = 0$ is not the defining equation for $Y_1(N)$ as $[N]P = 0$ implies only that $P$ has order $d$ for some $d|N$.  The polynomial $f_N(b,c)$ will have as irreducible factors defining equations for 
$Y_1(d)$ for $d \mid N$, $d > 3$.   However a simple Moebius inversion will furnish such an equation. Although we do not explicitly write down equations for $Y_1(N)$ in our algorithm, one could do so with relative ease. A more sophisticated version of this computation has been undertaken by Andrew Sutherland \cite{Sutherland}.

\medskip
\begin{example} We computed $4(0,0)$ and $2(0,0)$ as part of our above examples, so $f_6(b,c) = b^2 - bc - bc^2$. The divisors of $6$ are $1,2,3$ and $6$. Thus by Moebius inversion, the equation for $Y_1(6)$ in the $(b,c)$-plane is $b-c-c^2$, a smooth plane curve. For higher $N$, $Y_1(N)$ is not naturally a plane curve -- e.g. the genus of $Y_1(N)$ will usually not be of the form 
$\frac{(d-1)(d-2)}{2}$ -- and so there will often be singularities in this plane model. \end{example}

We note in general that if $d \ge 3$ and $d\mid N$ then $f_d \mid f_N$. This is easy to see using the group law when $d\ge 4$. For $d=3$ we can see this by computing on the elliptic curve $E(b,c)$ over the function field ${\QQ(b,c)}$. Namely, if $(x,y)$ is any nonidentity point it is possible to compute that $[3](x,y) \pm (0,0)$ has $x$-coordinate with $b$-adic valuation 1. Therefore $b$ but not $b^2$ divides $f_N(b,c)$ when $3\mid N$. Therefore performing a Moebius inversion on the $f_N$ furnishes a factorization $f_N = \prod_{\stackrel{d \mid N}{d \ge 3}} \phi_d$ where if $N\ge 4$ then $\phi_N(b,c)$ is a defining equation for $Y_1(N)$ in the $(b,c)$-plane.

\subsection{Complex multiplication and bounds on $j$-invariants} If $E(b,c)$ is a CM elliptic curve defined over a number field of degree $d$, then its $j$-invariant $j(b,c)$ must lie in a number field of degree dividing $d$. The degree of $\QQ(j(b,c))$ is equal to the class number of End $E$, which is an order in an imaginary quadratic field. 

\begin{theorem} (Heilbronn, 1934) \cite{Heilbronn} For any positive integer $d$, there are only finitely many imaginary quadratic fields with class number $d$. \end{theorem}


\begin{corollary} For any positive integer $d$, there are only finitely many imaginary quadratic orders $\OO$ such that $h(\OO) \le d$. \end{corollary}

\begin{proof} Since every quadratic order $\OO$ must be of the form $\ZZ + f\OO_K$, this follows from Gauss's class number formula \cite[Thm 7.24]{Cox}. 
\end{proof} 

\medskip

\begin{example}\label{full7}: We find the least possible degrees for an elliptic curve over a number field $K$ with $7$-torsion and $j$-invariant 0. If we have such a curve $E$, we can find a pair $(b,c)\in K^2$ such that $E \cong E(b,c)$.  Since $j(b,c) = 0$, we have 
\begin{equation}\label{jzeq}
16b^2+8b(1-c)(c+2)+(1-c)^4 = 0,
\end{equation}

and since $(0,0)$ is a nontrivial $7$-torsion point, we have

\begin{equation}\label{f7eq}
b^2-bc-c^3 = 0.
\end{equation}

The real solutions to Equation \ref{jzeq} in the $(b,c)$ affine plane may be seen in Figure \ref{jzerofig}, and Equation \ref{f7eq} in Figure \ref{f7fig}.

\begin{figure}
\centering
\includegraphics[width = \textwidth]{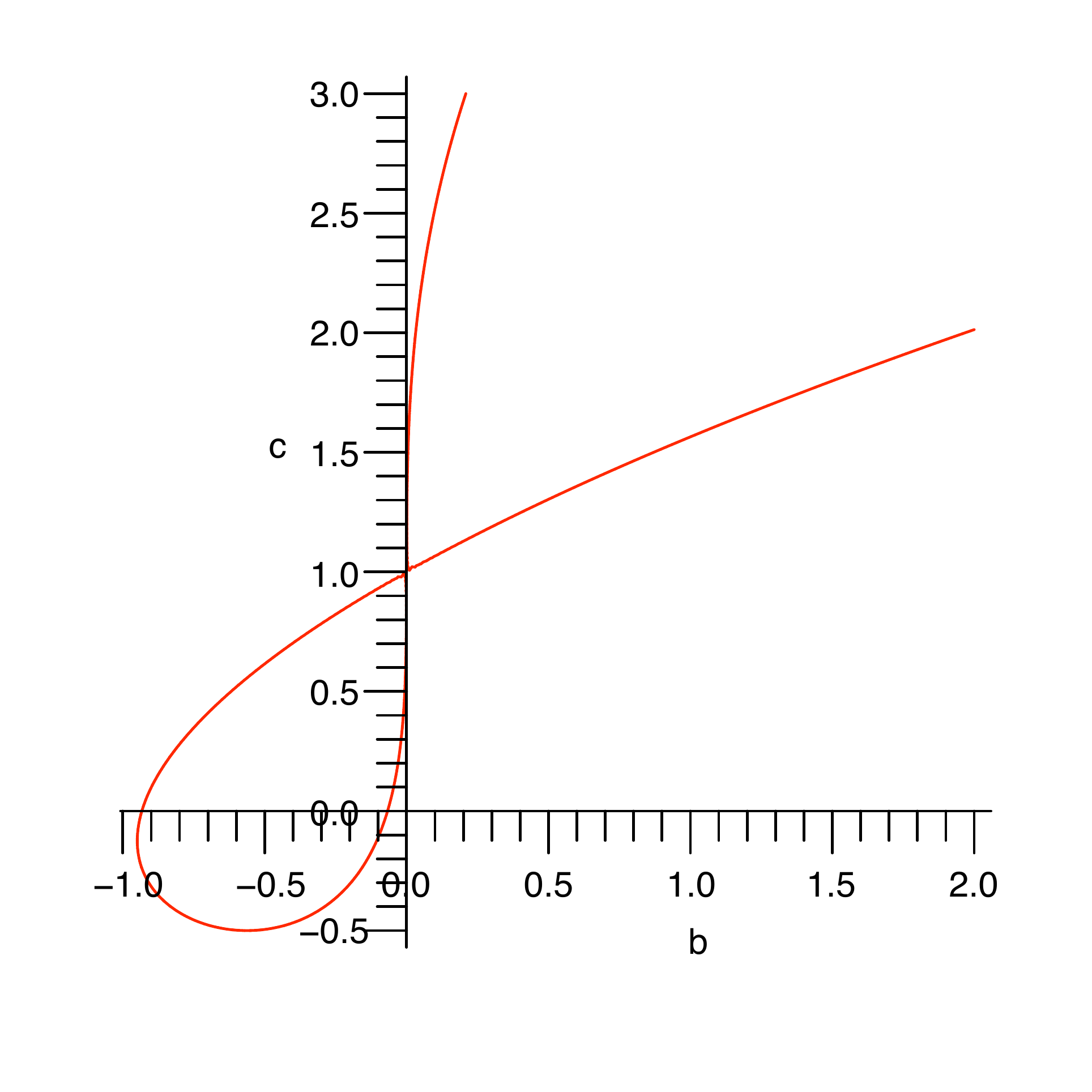}
\caption{The real $(b,c)$ such that $E(b,c)$ has $j$-invariant 0.}\label{jzerofig}
\end{figure}

\begin{figure}
\centering
\includegraphics[width = \textwidth]{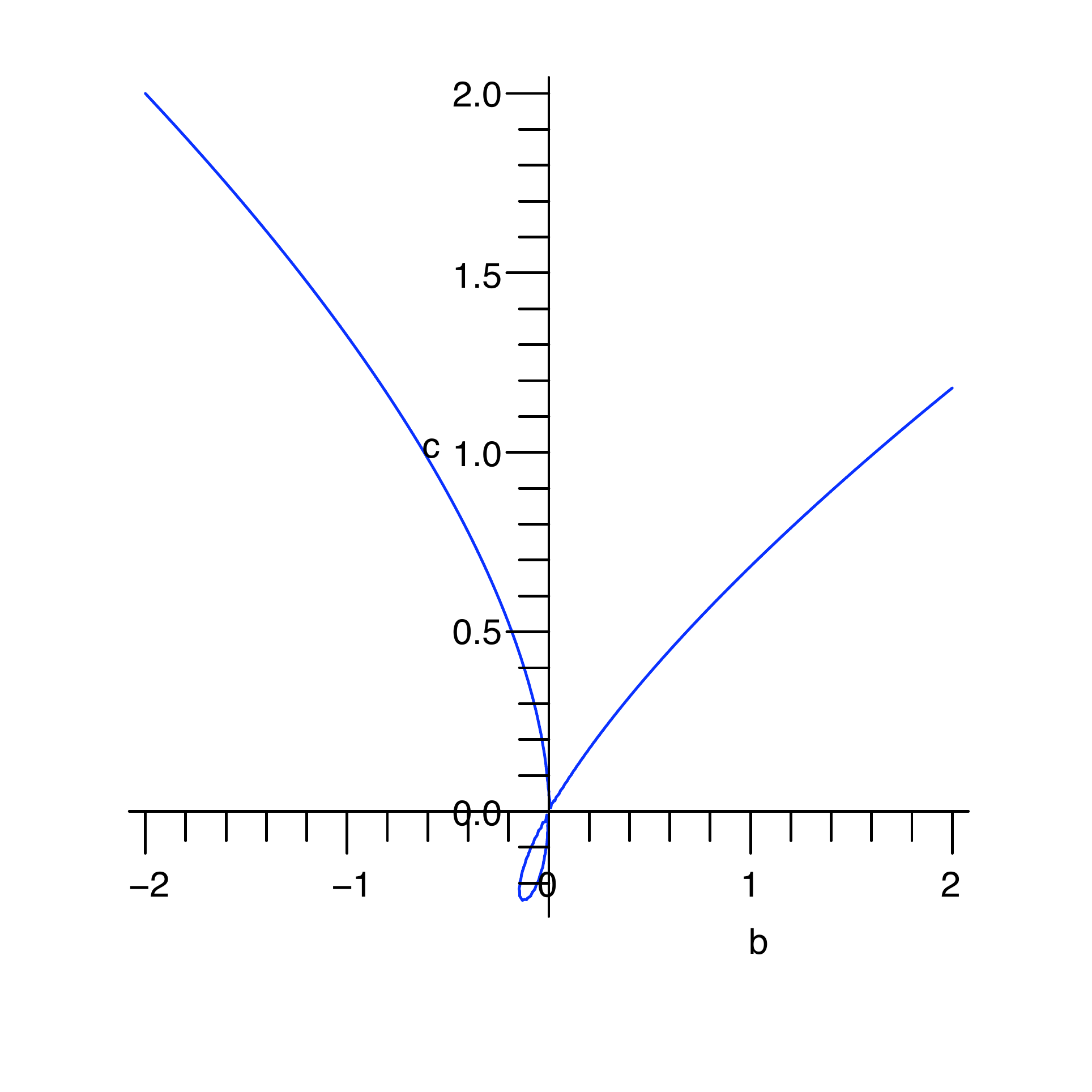}
\caption{The real $(b,c)$ such that $(0,0)$ is a $7$-torsion point on $E(b,c)$.}\label{f7fig}
\end{figure}

The resultant of these two polynomials with respect to $c$ is
\[
(b^2 + b + 1)(b^6 - 325b^5 + 5518b^4 + 3655 b^3 + 718 b^2 + 51 b + 1).
\]
The roots of this \textit{Kubert resultant} identify the intersection points of our two affine curves, as shown in Figure \ref{overlayfig}. We should note here that the first irreducible factor has no real roots. Instead, the $b$-coordinates of the intersection points we see are four of the six real roots of the second factor. In any case, looking at the first irreducible factor over $\QQ$, we see that we can take $b = \zeta_3$. 

\begin{figure}
\centering
\includegraphics[width = \textwidth]{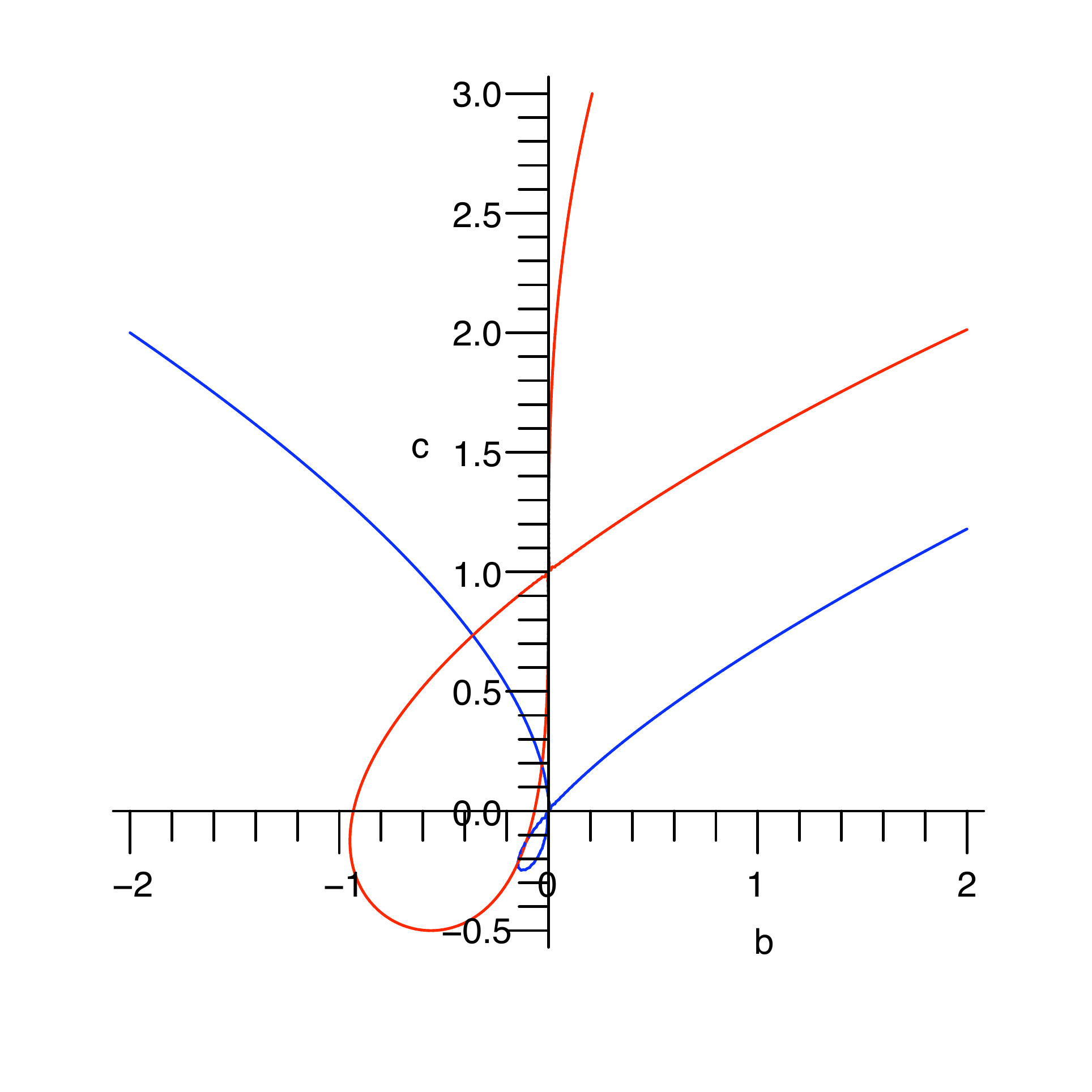}
\caption{The real $(b,c)$ such that $(0,0)$ is a $7$-torsion point on $E(b,c)$ with $j$-invariant 0.}\label{overlayfig}
\end{figure}

We plug in $\zeta_3$ for $b$ in the above polynomials and compute the greatest common divisor, which is $c+1$. So the elliptic curve $E(\zeta_3,-1)$ has a $7$-torsion point over $\QQ(\zeta_3)$. That is, on the curve
\[
y^2 + 2xy - \zeta_3 y = x^3 - \zeta_3 x^2,
\]
the point $(0,0)$ is a $7$-torsion point.\footnote{The reader who prefers standard Weierstrass models may verify that the origin corresponds to the $7$-torsion point $(12(1-\zeta_3), -108 \zeta_3)$ on the isomorphic elliptic curve 
$y^2 = x^3 - (1296\zeta_3 + 6480)$.}
Moreover, this curve acquires full $7$-torsion over the degree-12 cyclotomic field $\QQ(\zeta_{21})$.  
\end{example}

We generalize the above construction as follows: Writing $j(b,c) = \dfrac{n_j(b,c)}{d_j(b,c)}$ as the quotient of two polynomials, we see that there is an elliptic curve $E(b,c)$ with $j$-invariant $j_0$ and $[N](0,0) =O$ if and only if $(b,c)$ satisfy the equations
\begin{align}
n_j(b,c) &= j_0 d_j(b,c) \label{curves} \\
f_N(b,c) &= 0 .\notag
\end{align}

If there are only finitely many pairs $(j_0,N)$ that we have to check, then since the resultant of these equations with respect to $c$ is a one-variable polynomial in $b$, there are only finitely many elliptic curves $E(b,c)$ over a small-degree number field with $j$-invariant $j_0$ and with $(0,0)$ an $N$-torsion point. To determine if $\ZZ/N\ZZ \oplus \ZZ/n\ZZ$ with $n \mid N$ is a torsion subgroup of an elliptic curve over a small-degree number field, we need only check the $n$-th {\em division polynomial} \cite[Exercise 3.7]{Silverman} to see if $E(b,c)$ acquires an additional $n$-torsion point over a small-degree number field. There are finitely many $n\mid N$ and for each such $n$, there are algorithms to compute the $n$-th division polynomial.

In this way, we see how a rough algorithm for enumerating torsion subgroups of CM elliptic curves presents itself. Fix a degree $d$, so that we aim to tabulate CM torsion subgroups over number fields of degree $d$. By Heilbronn's theorem, there are only finitely many $j$-invariants of elliptic curves with Complex Multiplication over all number fields of degree at most $d$. By Merel's bound, we have only finitely many possible torsion subgroups to check. Since there are only finitely many $j_0$ and $N$, the procedure described above terminates for each $d$.

We note here that Merel's bound is quite large and often impractical. We mention it only to note that the above procedure terminates for any finite number of $j$-invariants, CM or not. In the CM case, we have much better bounds to consider.

\subsection{Possible torsion of CM elliptic curves} Let $E$ be an elliptic curve over a number field $F$ with CM. If $E(F)$ contains an $N$-torsion point, then the size of $N$ is severely restricted by the degree of $F$; the following theorems of Silverberg and Prasad-Yogananda can be used to give an explicit upper bound on $N$.

\begin{theorem}\label{SPYBounds} (Silverberg, Prasad-Yogananda) Let $E$ be an elliptic curve over a number field $F$ of degree $d$, and suppose that $E$ has CM by the order $\OO$ in the imaginary quadratic field $K$. Let $e$ be the exponent of the torsion subgroup of $E(F)$. Then 

(a) $\varphi(e) \le w(\OO) d$.

(b) If $K \subseteq F$, then $\varphi(e) \le w(\OO) d/2$.

(c) If $K \nsubseteq F$, then $\varphi(\# E(F)[\tors]) \le w(\OO) d$.
\end{theorem}

\begin{proof} See \cite{Silverberg}, \cite{PY}. It can be deduced from Silverberg's work that all above occurrences of $w(\OO)$ may be replaced with $w(\OO)/h(\OO)$.\end{proof}

\smallskip

We will refer henceforth to the bounds obtained from the above theorem and proof as the {\em SPY bounds}. Using merely the bound of part (a) and the well-known inequality $\sqrt{N} \le \phi(N)$ for $N\ge 7$, we see that we need only consider values of $N$ that are at most $w(\OO)^2d^2$. The SPY bounds also lead us to expect that the largest torsion subgroups occur when $w(\OO)$ is largest, namely when $j=0,1728$.

Any bound on the exponent above gives a bound on the size of the torsion subgroup. If the exponent of $E(F)[\tors]$ is at most $N$, since $E(\mathbf{C})[\tors] \cong (\QQ/\ZZ)^2$ \cite[Corollary V.1.1]{Silverman}, we have $\# E(F)[\tors] \le N^2$. In fact, there exist integers $n\mid N$ such that $E(F)[\tors]\cong \ZZ/N\ZZ \oplus \ZZ/n\ZZ$. Moreover in that case, the Weil Pairing \cite[\S III.8]{Silverman} shows that $F \supset \QQ(\zeta_n)$ and thus $\varphi(n) \mid [F:\QQ] =d$.

In the case that $E$ has CM by $\OO$, note that $j(E) \in F$ so that $\QQ(j(E)) \subset F$ and thus $h(\OO) \mid [F:\QQ] = d$. Therefore, let $\deg = d/h(\OO)$. The strengthening of the SPY bounds as noted in the proof of Theorem \ref{SPYBounds} implies that if $e$ is the exponent of $E(F)[\tors]$ then $\varphi(e) \le w(\OO)\deg$. Note also that if $\deg$ is odd, $K\not\subset F$ since $K$ and $\QQ(j(E))$ are linearly disjoint. 

If $\deg =2$ then we may assume that $j \ne 0,1728$ because the possible groups in that case have already been determined \cite{Zimmer1}. Thus $w(\OO) =2$, hence either $E(F)[\tors]$ is among the 12 possible torsion subgroups $G$ such that $\varphi(\# G) \le 4$ or $F$ is the compositum of $\QQ(j(E))$ with $K$, otherwise known as the ring class field of $\OO$. In the latter case, we have the following.

\begin{theorem}(Parish) Let $\OO$ be an imaginary quadratic order, $j$ the $j$-invariant of an elliptic curve with CM by $\OO$, $L = \QQ(j)$ and $H$ the ring class field of $\OO$. Then if $E$ is an elliptic curve defined over $H$ with CM by $\OO$ then $E(K)[\tors]$ contains only points of order $1,2,3,4,$ or $6$. Moreover, if $E$ is defined over $L$ then $E(L)[\tors]$ can only be isomorphic to one of $0,\ZZ/2\ZZ,\ZZ/3\ZZ,\ZZ/4\ZZ,\ZZ/6\ZZ,$ or $\ZZ/2\ZZ \oplus \ZZ/2\ZZ$.  \end{theorem}

\begin{proof} \cite[\S VI]{Parish}. \end{proof}

Much finer information is available in Parish's paper. Except for $j= 0$ and $\ZZ/3\ZZ\oplus\ZZ/3\ZZ$, each torsion subgroup $G$ which is possible over a ring class field has $\varphi(\#G) \le 4$. Note that as a further consequence, if $E$ is an elliptic curve with CM by $\OO$ over a number field $F$ and $[F:\QQ] = h(\OO)$, then the only possible torsion subgroups are those found in degree 1. By the Weil pairing, if $E(F)[\tors] \cong \ZZ/N\ZZ \oplus \ZZ/n\ZZ$ with $n\mid N$ then $\varphi(n) \mid d$. In most cases though, $\varphi(n)\mid \deg$. We can however determine exactly the intersection of $\QQ(\zeta_n)$ with $\QQ(j(\OO))$ and thus get closer to the ideal that $\varphi(n) \mid \deg$. 

Let $H$ denote the ring class field of $\OO$ and $G$ the maximal abelian sub-extension of $H$ over $\QQ$, which is necessarily multi-quadratic \cite[\S 6]{Cox}. Hence $G'$, the intersection of $\QQ(j(\OO))$ with $G$, must also be multiquadratic. Since any abelian extension must be contained in some $\QQ(\zeta_m)$ \cite[Theorem 8.8]{Cox}, the intersection of $\QQ(\zeta_n)$ with $\QQ(j(\OO))$ must be contained in $G'$. This $G'$ may be numerically determined via discriminants, but it is not computationally difficult to simply list the discriminants of the quadratic subfields of $\QQ(j(\OO))$, which are all necessarily real. If $\Delta$ is a discriminant of a real quadratic field $K$, then $K \subset \QQ(\zeta_n)$ if and only if $\Delta \mid n$ \cite[Example V.3.11]{Milne}. Finally, we may determine inductively that if $M$ is a multi-quadratic field extension of degree $2^m$, then the number of quadratic sub-extensions is $2^m -1$.

\begin{function}\label{DegQJCyc} (\textsf{CyclotomicIntersectionDegree}) Let $\OO$ be an imaginary quadratic order and $n$ a positive integer.
\begin{enumerate}
\item Let $L = \{ \disc(K) : [K:\QQ]=2, K \subset \QQ(j(\OO))\}$, the discriminants of the quadratic subfields of the multi-quadratic field $G'$ above.
\item Let $M = \{ D :  D \in L, D \mid n\}$, the discriminants of the quadratic subfields of $G' \cap \QQ(\zeta_n)$.
\item Return \# $M + 1$.
\end{enumerate}
\end{function}

These steps restrict the groups which could possibly occur as torsion subgroups of an elliptic curve with CM by $\OO$. We combine these steps into a function, which takes as input an imaginary quadratic order $\OO$, a degree $d$, and a list of integers $N$ which could be the exponent of a torsion subgroup of an elliptic curve $E$ over a number field $F$ with CM by $\OO$. The output of this function is a list of finite abelian groups $G$ such that $E(F)[\tors] \cong G$ for an elliptic curve $E$ with CM by $\OO$.

\begin{function}\label{PossGroups} (\textsf{PossibleGroups}) Let $d\in \ZZ_{>1}$, $\OO$ an imaginary quadratic order such that $h(\OO) \mid d$, and $L$ a list of positive integers $N$.
\begin{enumerate}
\item Set $L' = \{ \ZZ/N\ZZ \oplus \ZZ/n\ZZ : N \in L, n \mid N\},$ $h = h(\OO)$, and $\deg = \dfrac{d}{h}$.
\item If $\deg =1$ then remove $\ZZ/N\ZZ \oplus \ZZ/n\ZZ$ from $L'$ unless $n=N =2$ or $n=1$ and $N \in \{1,2,3,4,6\}$.
\item If $\deg =2$ then remove $\ZZ/N\ZZ \oplus \ZZ/n\ZZ$ from $L'$ unless either ($\OO \cong \ZZ[\zeta_3]$ and $(N,n) = (3,3)$) or $\varphi(Nn) \le 4$.
\item If $\deg >1$ is odd then remove $\ZZ/N\ZZ \oplus \ZZ/n\ZZ$ from $L'$ unless $\varphi(Nn) \le w(\OO)\deg$ and $\varphi(n) \mid  d$.
\item If $\deg >2$ is even then remove $\ZZ/N\ZZ \oplus \ZZ/n\ZZ$ from $L'$ unless $\varphi(n) \mid \deg \times\textsf{CyclotomicIntersectionDegree}(\OO,n)$ (Function \ref{DegQJCyc}).
\item Return $L'$.
\end{enumerate}
\end{function}

In Function \ref{PossGroups}, you will of course get the best results when the list $L$ is made up of integers $N$ which can be an order of a torsion point on an elliptic curve $E$ with CM by $\OO$. Necessarily then, $\varphi(N) \le \dfrac{w(\OO) d}{h(\OO)} = w(\OO)\deg$ by the SPY bounds. We also have another tool for ruling out possible orders of torsion.

\begin{theorem}\label{CCSBounds} Let $\OO$ be an imaginary quadratic order of discriminant $D$ and let $D_0$ be the discriminant of the field $K =\QQ(\sqrt D)$, so that $D = f^2 D_0$. If $p \nmid D$ is an odd prime then let $\left(\dfrac{\cdot}{p}\right)$ denote the Legendre symbol at $p$. If $E$ is an elliptic curve over a number field of degree $d$ with CM by $\OO$ with a point of order $p$ then we have the following.
\begin{itemize}
\item If $\left(\dfrac{D}{p}\right) = 1$ then $(p-1)h(\OO_K) \mid 2dw(\OO_K)$.
\item If $\left(\dfrac{D}{p}\right) = -1$ then $(p^2-1)h(\OO_K) \mid 2dw(\OO_K)$.
\end{itemize}
\end{theorem}
\begin{proof} This was directly proven for $D = D_0$ \cite[Theorem 2]{CCS}, and can be extended to the case $p\nmid D$ \cite[Proposition 25]{CCS}.\end{proof}

In this way, we can additionally remove large primes from the divisors of possible exponents. Starting from a list of integers up to $w(\OO)\deg$, we can then very quickly sieve out impossible torsion exponents. For $j=0$, performing the above procedure takes $\dfrac{1}{100}$ of one second to find $$[ 2, 3, 4, 5, 6, 7, 8, 9, 10, 12, 13, 14, 15, 16, 18, 20, 21, 24, 26, 28, 30, 36, 42 ]$$ as a list of possible torsion exponents over a number field of degree 2. 

\begin{function}\label{PossibleN} (\textsf{PossibleExponents}) Let $\OO$ be an imaginary quadratic order and let $\deg$ be a positive integer.
\begin{enumerate}
\item Let $L$ be the list of positive integers $N$ such that $\varphi(N) \le w(\OO) \deg$.
\item Let $L'$ be the set of integers $N\in L$ such that if $p\mid N$ is prime then $p$ satisfies the divisibility relations in Theorem \ref{CCSBounds} for $d = h(\OO)\deg.$
\item Return $L'$.
\end{enumerate}
\end{function}

Note however that this list is still far too large a list to use in Function \ref{PossGroups}. We apply a sieve to this list, using resultants as in Example \ref{full7}, where we showed that 7-torsion occurred over a number field of degree 2 for $j=0$. We note especially that the {\em Kubert Degree Sequence} for $j=0$ and $N=7$, or the sequence of degrees of irreducible factors of the resultant, is $[2,6]$. On the other hand, the Degree Sequence for $j=0$ and $N = 14$ is $[6,18]$. Therefore we may eliminate $14,28$, and $42$ from our list of possible torsion exponents because 14-torsion is not possible for $j=0$ over a number field of degree not divisible by 6. Computing this Degree Sequence takes $0.03$ seconds. If we recursively perform this sieve, it takes $0.24$ seconds to find that the torsion exponents which occur for $j=0$ over a number field of degree 2 are $$[ 2, 3, 4, 6, 7 ].$$ This may seem like a relatively short amount of time to be worried about, but for $j=0$ and a number field 
of degree 6 it takes $69.95$ seconds to find $[ 2, 3, 4, 6, 7, 9, 14, 19 ]$ as the list of torsion exponents. For degree 12 it takes over an hour. We describe this process, along with the adjustment we have to make for $\OO$ with larger class numbers in the following function.

\begin{function}\label{FullResSieve} (\textsf{SievedTorsion}) Let $\OO$ be an imaginary quadratic order and let $\deg$ be a positive integer.
\begin{enumerate}
\item Let $L = \textsf{PossibleExponents}(\OO,\deg)$. (Function \ref{PossibleN})
\item For $N \in \textsf{PossibleExponents}(\OO,\deg)$ such that $N \in L$ and $N \ge 4$:
\begin{itemize}
\item Let $DegSeq$ be the sequence of integers $Degree(f)h(\OO)$ where $f$ is an irreducible factor of the resultant corresponding to $N$-torsion on elliptic curves with CM by $\OO$.
\item Unless $m \mid h(\OO)\deg$ for some $m \in DegSeq$, remove all multiples of $N$ from $L$.
\end{itemize}
\item Return $L$.
\end{enumerate}
\end{function}

We structure our computation this way to minimize the number of times that we need to compute multivariate resultants. While straightforward and much quicker than computing torsion subgroups of elliptic curves, the computation of multivariate resultants is \textsf{NP}-Hard \cite{Res}. The memory demands for computing resultants over large degree number fields can also be quite substantial. All told, the longest computation of torsion subgroups occurred in degree 12. Computing the lists of possible torsion subgroups of CM elliptic curves over a number field of degree 12 for each possible quadratic order $\OO$ using the above procedure took over 10 hours.

\section{Ruling out Torsion Subgroups of Elliptic Curves}

Suppose we are given a finite group $G \cong \ZZ/N\ZZ \oplus \ZZ/n\ZZ$ and we want to test whether it could be a torsion subgroup of an elliptic curve $E$ over a number field $F$ of degree dividing $d$ with CM by an imaginary quadratic order $\OO$. If there is such an elliptic curve such that $E(F)[\tors]\cong G$, then we can find $b,c\in F$ such that $E \cong E(b,c)$, where the point $(0,0)$ is a point of order $N$. Conversely if we have $b,c\in F$ such that $E(b,c)$ has CM by $\OO$ and $(0,0)$ is a point of order $N$, it is not necessarily the case that $E(b,c)(F)[\tors]\cong G$. The first and easiest way for this to fail is if $E(b,c)(F)[\tors] \supsetneq G$.

\begin{example} The resultant whose roots are the $b$ such that $(0,0)$ is a 5-torsion point on $E(b,c)$ with CM by $\ZZ[\zeta_4]$ is $$(x^2 +1)^2(x^4 - 18x^3 + 74x^2 + 18x + 1)^2.$$ However, any elliptic curve over a number field $F$ with CM by $\ZZ[\zeta_4]$ has a rational 2-torsion point for trivial reasons\footnote{For instance, put your elliptic curve in short Weierstrass form.}. Therefore if we search for $\ZZ/5\ZZ$ as a torsion subgroup over a degree 2 field, we find $\ZZ/10\ZZ$ as the torsion subgroup of $E(\zeta_4,\zeta_4)$.\end{example}

It of course may also happen that $G\subsetneq E(F)[\tors]$ and that they have the same exponent. The more typical situation is that $E(F)[\tors] \subset G$. In that case, we have to check to see if there is an extension field $L$ of $F$ of degree still dividing $d$ such that $E(L)[\tors] \cong G$.

To rule this out, there are many options. Of course, we may compute all elliptic curves with CM by $\OO$ and with an $N$-torsion point using the Kubert resultant method of Example \ref{full7}, base extend each of these elliptic curves by roots of their $n$-th division polynomials and then compute the torsion subgroups of all those elliptic curves. For running time reasons, it is preferable to rule this out before ever computing an elliptic curve or especially a torsion subgroup. Although there are many ways to compute a torsion subgroup of an elliptic curve over a number field, almost all of them involve reducing an elliptic curve modulo various primes in order to take advantage of Schoof's algorithm \cite{Schoof}. Unfortunately, irreducible factors of Kubert resultants often have non-integral coefficients, making this process slow and un-supported in some computer algebra systems. 
Even for computer algebra systems like \texttt{magma v2-18.3} with robust support for elliptic curves over number fields 
given by non-integral polynomials, it can be very time- and memory-consuming to compute torsion subgroups over large degree number fields.

A crucial step is thus a variant of Step (5) of Function \ref{PossGroups}. If $\ZZ/N\ZZ \oplus \ZZ/n\ZZ \cong E(L)[\tors]$ with $n\mid N$ for some field extension $L$ of $F$, then we must have $\QQ(\zeta_n) \subset L$. 
Numerically we have done almost everything to numerically rule out the possibility that there is some field $L$ of degree dividing $d$ which contains both $F$ and $\QQ(\zeta_n)$. 
Now that we have computed $F$ explicitly via the Kubert resultant, we can compute the compositum of $F$ with $\QQ(\zeta_n)$ and its degree over $\QQ$. 
If this degree does not divide $d$, then we can not base extend $F$ to $L$ and obtain $E(L)[\tors] \cong G$. Moreover, we have ruled $G$ out without computing any torsion on $E$.

\begin{example} Let $\OO = \ZZ\left[\dfrac{1 + 3\sqrt{-11}}{2}\right]$, let $d=12$, and let $G \cong \ZZ/9\ZZ \oplus \ZZ/9\ZZ$. Since $h(\OO) =2$, the Kubert Degree Sequence for $\OO$ and $N = 9$ is $[6,12,54]$. An elliptic curve with CM by $\OO$ over the number field defined by the first irreducible factor or a degree two extension thereof cannot have torsion subgroup $G$ by a quick standard computation. Just computing the torsion subgroup over the number field $F$ given by the second irreducible factor ran for several days before quitting due to a lack of memory. While we could not rule out $G$ as a torsion subgroup without computing with $F$, we found that the compositum of $F$ with $\QQ(\zeta_9)$ has degree 36 over $\QQ$ and therefore $G$ is not a torsion subgroup of an elliptic curve with CM by $\OO$ over a number field of degree 12.
\end{example}

We now describe the procedure for saying that a group $G$ which could have been produced by Functions \ref{PossibleN} and \ref{PossGroups} in fact cannot appear as the torsion subgroup of an elliptic curve $E$ with CM by $\OO$ over a number field $L$ of degree dividing $d$. We describe this procedure as a function which either returns \textsf{True} if $G$ can be ruled out or \textsf{False} if $G$ can occur, along with an elliptic curve $E$ over a number field $L$ of degree dividing $d$.

\begin{function}\label{RuledOut}(\textsf{RuledOut})
Let $G \cong \ZZ/N\ZZ \oplus \ZZ/n\ZZ$, let $d$ be a positive integer and let $\OO$ be an imaginary quadratic order.
\begin{enumerate}
\item Compute the Kubert Resultant whose roots are the $b\in \overline\QQ$ such that $E(b,c)$ has CM by $\OO$ and $(0,0)$ is a point of exact order $N$. Factor it as $\displaystyle \prod_{i=1}^g f_i$.
\item\label{RuledOutBeginning} If $Degree(f_i)h(\OO) \mid d$ then let $F_i$ denote the number field given by $f_i$, generated over $\QQ(j(\OO))$ by $b_i$. Let $c_i$ be the element of $F_i$ (or possibly an extension) such that $E(b_i,c_i)$ is our CM elliptic curve. Compute the compositum of $F_i$ with $\QQ(\zeta_n)$ and let $d_i$ be its degree over $\QQ$.
\item\label{RuledOutError} If $c_i \not\in F_i$ then raise an error.
\item If $d_i \mid d$ then let $T_i$ be the torsion subgroup of $E(b_i,c_i)$. If $T_i \cong G$ then Return \textsf{False}, $E(b_i,c_i)_{F_i}$.
\item If $[F_i:\QQ] \ne d$ then it may be possible to base extend $E(b_i,c_i)$ to obtain $G$ as a torsion subgroup. 
\item If $T_i$ is a subgroup of $G$ with the same exponent then $T_i \cong \ZZ/N\ZZ \oplus \ZZ/n'\ZZ$ where $n' \mid n \mid N$. Compute the $n$-th division polynomial of $E(b_i,c_i)$, perform Moebius inversion to obtain a polynomial whose roots are $x$-coordinates of points of exact order $n$, and factor that polynomial as $\displaystyle\prod_{j=1}^m p_j$.
\item If $Degree(p_j) \mid \dfrac{d}{[F_i:\QQ]}$ then let $L_{i,j}$ be the number field given by $p_j$, generated over $F_i$ by the $x$-coordinate $a_j$. Let $g = y^2 + (a_j(1-c_i) - b_i)y + (b_ia_j^2 - a_j^3)$, the polynomial whose roots in $\overline \QQ$ are the $y$-coordinates of the points on $E(b_i,c_i)$ with $x$-coordinate $a_j$. Let $n_g$ be the number of irreducible factors over $L_{i,j}$ of $g$ and let $e_j = Degree(p_j)\dfrac{2}{n_g}$.
\item If $e_j \ne 1$ and $e_j \mid \dfrac{d}{[F_i:\QQ]}$ then let $M_{i,j}$ be the field given by the polynomial $g$. Let $T_{i,j}$ be the torsion subgroup of the base change of $E(b_i,c_i)$ to $M_{i,j}$.
\item\label{RuledOutEnding} If $T_{i,j} \cong G$ then Return \textsf{False}, $E(b_i,c_i)_{M_{i,j}}$.
\item If for all possible $i$ and $j$ there is some ``If $\ldots$''' statement which begins one of Steps (\ref{RuledOutBeginning})-(\ref{RuledOutEnding}) besides Step \ref{RuledOutError} which is false, then Return \textsf{True}.
\end{enumerate}
\end{function}

We note that in Step \ref{RuledOut}(\ref{RuledOutError}), it is possible that $c_i \not \in F_i$ and thus we must have an error-raising statement. However, as one may intuit from Figure \ref{overlayfig}, the probability that two intersection points in the $(b,c)$-plane have the same $b$ value is zero by the properties of the Zariski topology. We now give an algorithm which produces all torsion subgroups of elliptic curves with CM over a number field of degree $d$.

\begin{algorithm}\label{FinalAlg}
Let $d$ be a positive integer and $L$ a list of finite groups which we know to be torsion subgroups of CM elliptic curves over some number field of degree dividing $d$.
\begin{enumerate}
\item Create an associative array or dictionary $A$, indexed by imaginary quadratic orders $\OO$ such that $h(\OO)\mid d$ and either $h(\OO) =1$ or $h(\OO) \ne d$. Let the $\OO$-th entry of $A$ be $$\mathsf{PossibleGroups}\left(d,\OO,\mathsf{SievedTorsion}\left(\OO, \dfrac{d}{h(\OO)}\right)\right).$$
\item\label{GroupsToBeRuledOut} Let $P$ be the union of all the sets $A(\OO)$ and let $R$ be $P - L$, the set of groups in $P$ which are not isomorphic to any element of $L$.
\item Iterate over $G\in R$.
\begin{itemize}
\item If $\mathsf{RuledOut}(G,d,\OO)$ returns \textsf{True} for all $\OO$ such that $G\in A(\OO)$, move onto the next group.
\item If not, append $G$ to $L$ and go to Step (\ref{GroupsToBeRuledOut}).
\end{itemize}
\end{enumerate}
\end{algorithm}

When Algorithm \ref{FinalAlg} is completed, $L$ is the complete list of possible torsion subgroups. If $d=2$, then Algorithm \ref{FinalAlg} takes 0.87 seconds to complete when starting with the list given by Zimmer, M\"uller and Stroher, and rules out only the group $\ZZ/5\ZZ$. If $d=12$, then if we start from Step (\ref{GroupsToBeRuledOut}) with a complete list $L$, the algorithm takes only 3.5 hours to complete for a total time of roughly 14 hours. Complete records of the ruling out computation may be found on \url{stankewicz.net/torsion}.

\section{Isomorphism classes of Torsion Subgroups of CM elliptic curves $E$}\label{lists}

\subsection{$K= \QQ$}

$$E(\QQ)[\tors] \in \{0,\ZZ/2\ZZ, \ZZ/3\ZZ, \ZZ/4\ZZ, \ZZ/6\ZZ, \ZZ/2\ZZ \oplus \ZZ/2\ZZ\}.$$

Examples of these are : 

\begin{center}
\begin{tabular}{ccc} 
Group & Elliptic Curve & $j$-invariant \\
0 & $y^2 = x^3 +2$ & 0 \\
$\ZZ/2\ZZ$  & $y^2 = x^3 -1$ & 0 \\
$\ZZ/3\ZZ$  & $y^2 = x^3 + 16$ & 0 \\
$\ZZ/4\ZZ$  & $y^2 = x^3 + 4x$ & 1728\\
$\ZZ/6\ZZ$  & $y^2 = x^3 + 1$ & 0 \\
$\ZZ/2\ZZ\oplus\ZZ/2\ZZ$  & $y^2 = x^3 - 4x$ & 1728 \end{tabular}\end{center}

\subsection{$K$ is a number field of degree 2.}

$$E(K)[\tors] \in \begin{cases} \ZZ/m\ZZ & \textrm{for } m = 1,2,3,4,6,7,10, \\ \ZZ/2\ZZ \oplus \ZZ/m \ZZ & \textrm{for } m=2,4,6, \textrm{ and} \\ \ZZ/3\ZZ\oplus \ZZ/3\ZZ. &\end{cases}$$

The only subgroups which do not occur over $\QQ$ are:

$$E(K)[\tors] \in \{\ZZ/7\ZZ,\ZZ/10\ZZ, \ZZ/2\ZZ \oplus \ZZ/4 \ZZ,\ZZ/2\ZZ \oplus \ZZ/6 \ZZ, \ZZ/3\ZZ\oplus \ZZ/3\ZZ\}$$

Examples of these are:
\begin{center}
\begin{tabular}{cccc} 
Group & Field Extension & Elliptic Curve & $j$-invariant \\
$\ZZ/7\ZZ$ & $\QQ(\zeta_3)$ & $E(\zeta_3,-1)$ & 0\\
$\ZZ/10\ZZ$ & $\QQ(\zeta_4)$ & $E(\zeta_4,\zeta_4)$ & 1728\\
$\ZZ/2\ZZ\oplus\ZZ/4\ZZ$ & $\QQ(\zeta_4)$ & $y^2 = x^3 + 4x$ & 1728\\
$\ZZ/2\ZZ\oplus\ZZ/6\ZZ$ & $\QQ(\zeta_3)$ & $y^2 = x^3 + 1$ & 0\\
$\ZZ/3\ZZ\oplus\ZZ/3\ZZ$ & $\QQ(\zeta_3)$ & $y^2 = x^3 + 16$ & 0 \end{tabular}\end{center}

\subsection{$K$ is a number field of degree 3.}

$$E(K)[\tors] \in \begin{cases} \ZZ/m\ZZ & \textrm{for } m = 1,2,3,4,6,9,14, \\ \textrm{and }&\ZZ/2\ZZ \oplus \ZZ/2 \ZZ.  \end{cases}$$

The only subgroups which do not occur over $\QQ$ are:

$$E(K)[\tors] \in \{ \ZZ/9\ZZ, \ZZ/14\ZZ \}.$$

Examples of these are:
\begin{center}
\begin{tabular}{cccc} 
Group & Defining Polynomial & Elliptic Curve & $j$-invariant \\ \hline
$\ZZ/9\ZZ$ & $b^3 - 99b^2 - 90b - 9$ & $E\left(b,\displaystyle\frac{-2b^2 + 318b - 75}{753}\right)$ & 0 \\
$\ZZ/14\ZZ$ & $b^3 + 5b^2 + 2/7b - 1/49$ & $E\left(b,\displaystyle\frac{133b^2 + 749b + 54}{167}\right)$ & -3375 \end{tabular}\end{center}

\subsection{$K$ is a number field of degree 4.}

$$E(K)[\tors]\in \begin{cases} \ZZ/m\ZZ & \textrm{for } m = 1,\dots,8,10,\\ &\hspace{1.5cm} 12,13,21 \\ \ZZ/2\ZZ \oplus \ZZ/m \ZZ & \textrm{for } m=2,4,6,8,10,\\ \ZZ/3\ZZ\oplus \ZZ/m\ZZ & \textrm{for } m=3,6,\\ \textrm{and } &\ZZ/4\ZZ\oplus \ZZ/4\ZZ. \end{cases}$$

The only subgroups which do not occur over $\QQ$ or a number field of degree 2 are:

$$E(K)[\tors]\in \left\{\begin{array}{c}\ZZ/5\ZZ, \ZZ/8\ZZ, \ZZ/12\ZZ,\ZZ/13\ZZ, \ZZ/21\ZZ,\\ \ZZ/2\ZZ\oplus \ZZ/8\ZZ, \ZZ/2\ZZ\oplus \ZZ/10\ZZ, \ZZ/4\ZZ\oplus \ZZ/4\ZZ, \ZZ/3\ZZ\oplus \ZZ/6\ZZ\end{array}\right\}.$$

Examples of these are:

\hspace{-1cm}\begin{tabular}{cccc}
Group & Field & Elliptic Curve & $j$ \\ \hline
$\ZZ/5\ZZ$ & $\dfrac{\QQ[b]}{(b^4 - 4b^3 + 46b^2 + 4b + 1)}$ & $E(b,b)$ & -32768 \\
$\ZZ/8\ZZ$ & $\dfrac{\QQ[b]}{(b^4 + 2b^3 + b^2 - b - \dfrac{1}{8})}$ & $E\left(b,\displaystyle\frac{8b^3 + 36b^2 + 46b + 3}{13}\right)$ \footnotemark  & 1728\\
$\ZZ/12\ZZ$ & $\dfrac{\QQ[b]}{(b^4 - 10b^3 + 24b^2 - 16b - 2)}$ & $E\left(b,\displaystyle\frac{-6b^3 + 52b^2 - 70b - 9}{7}\right)$ & 0\\
$\ZZ/13\ZZ$ & $\dfrac{\QQ[b]}{(b^4 + 4b^3 + 78b^2 + 13b + 1)}$ & $E\left(b,\displaystyle\frac{16b^3 + 44b^2 + 1354b + 45}{483}\right)$ & 0\\
$\ZZ/21\ZZ$ & $\dfrac{\QQ[e]}{(e^{4} - e^{3} + 2 e + 1)}$ & 
$y^2 = x^3 - \left(\begin{array}{c}371952 e^{3} + \\3373488 e^{2} + \\  3777840 e + \\ 1228608\end{array}\right)$
& 0\\
$\left(\dfrac{\ZZ/2\ZZ\oplus}{\ZZ/8\ZZ}\right)$ & $\dfrac{\QQ[b]}{(b^4 - 4b^3 +  4b^2 - b - 1/8)}$ & $E\left(b,32b^3 - 108b^2 + 58b + 6\right)$ & 287496\\
$\left(\dfrac{\ZZ/2\ZZ\oplus}{\ZZ/10\ZZ}\right)$ & $\dfrac{\QQ(\zeta_4)[x]}{(x^2 - \zeta_4x -\dfrac{\zeta_4}{2})}$ & $E(\zeta_4,\zeta_4)$ & 1728 \\
$\left(\dfrac{\ZZ/4\ZZ\oplus}{\ZZ/4\ZZ}\right)$ & $\dfrac{\QQ[x]}{(x^4 + 1)}$ & $E(-1/8,0)$ & 1728 \\
$\left(\dfrac{\ZZ/3\ZZ\oplus}{\ZZ/6\ZZ}\right)$ & $\QQ(\sqrt 3, \sqrt{-3})$ & $E\left(\dfrac{6\sqrt 3 + 10}{3},\dfrac{2\sqrt 3 + 3}{3}\right)$ & 1728
\end{tabular}
\footnotetext{Although $j(\ZZ[2i]) = 287496$ and $j(\ZZ[i]) = 1728$ and so it would be reasonable to expect two curves over number fields of the same degree with those $j$-invariants and respective torsion subgroups $\ZZ/2\ZZ \oplus \ZZ/8\ZZ$ and $\ZZ/8\ZZ$ to be isogenous, this is not the case. Indeed the two fields are not isomorphic.}

\subsection{$K$ is a number field of degree 5.}

$$E(K)[\tors]\in \{0,\ZZ/2\ZZ ,\ZZ/3\ZZ,\ZZ/4\ZZ,\ZZ/6\ZZ,\ZZ/11\ZZ, \ZZ/2\ZZ \oplus \ZZ/2 \ZZ \}$$

The only subgroup which does not occur over $\QQ$ is $\ZZ/11\ZZ$, which occurs over the maximal real subfield of $\QQ(\zeta_{11})$. This occurs for $j = -32768$ in $E(b,c)$ with the following quantities. Hereon, unless otherwise stated, elliptic curves will be given by the values of $b$ and $c$.

\begin{tabular}{ccc}
 Field Extension & $b$ &$c$ \\ \hline
 $\dfrac{\QQ[e]}{(e^{5} - e^{4} - 4 e^{3} + 3 e^{2} + 3 e - 1)}$ & $-7 e^{4} - 2 e^{3} + 16 e^{2} - e - 1$ & $2 e^{3} - 4 e + 1$ 
\end{tabular}

\subsection{$K$ is a number field of degree 6.}

$$E(K)[\tors]\in \begin{cases} \ZZ/m\ZZ & \textrm{for } m = 1,2,3,4,6,7,9,10,\\ & \hspace{1.5cm}14,18,19,26, \\ \ZZ/2\ZZ \oplus \ZZ/m \ZZ & \textrm{for } m=2,4,6,14,\\ \ZZ/3\ZZ\oplus \ZZ/m\ZZ & \textrm{for } m = 3,6,9,\\ \textrm{and } &\ZZ/6\ZZ \oplus \ZZ/6\ZZ. \end{cases}$$

The only subgroups which do not occur over $\QQ$ or a number field of degree 2 or 3 are:

$$E(K)[\tors]\in \left\{\begin{array}{c}\ZZ/18\ZZ,\ZZ/19\ZZ,\ZZ/26\ZZ \\ \ZZ/2\ZZ\oplus \ZZ/14\ZZ, \ZZ/3\ZZ\oplus \ZZ/6\ZZ,\ZZ/3\ZZ\oplus \ZZ/9\ZZ, \ZZ/6\ZZ\oplus\ZZ/6\ZZ \end{array} \right\}.$$

Examples of these are:


$\begin{array}{l|l}
Group & \ZZ/18\ZZ \\ \hline
j, Field & 8000,\QQ[e]/(e^{6} - 2 e^{5} + 3 e^{4} - 2 e^{3} + 2 e^{2} + 1) \\ \hline
b & \frac{1}{9}\left(\begin{array}{c}28 e^{5} - 79 e^{4} + 86 e^{3} - 30 e^{2} + 11 e - 31
\end{array}\right) \\ \hline 
c & \frac{1}{3}\left(\begin{array}{c}-8 e^{5} + 7 e^{4} - e^{3} - 8 e^{2} - 9 e - 5
\end{array}\right)
\end{array}$


$\begin{array}{l|l}
Group & \ZZ/19\ZZ \\ \hline
j,Field & 0,\QQ[e]/(e^{6} + e^{4} - e^{3} - 2 e^{2} + e + 1) \\ \hline
b & 2 e^{5} - e^{4} + 2 e^{3} - 4 e^{2} + 2 \\ \hline 
c & 2 e^{5} - 2 e^{4} + 4 e^{3} - 4 e^{2} - 2 e + 3
\end{array}$


$\begin{array}{l|l}
Group & \ZZ/26\ZZ \\ \hline
j, Field & 1728,\QQ[e]/(e^{6} - e^{4} + 2 e^{3} - 2 e + 1) \\ \hline
b & \frac{1}{13}\left(\begin{array}{c}53 e^{5} + 6 e^{4} - 70 e^{3} + 75 e^{2} + 5 e - 136
\end{array}\right) \\ \hline 
c & \frac{1}{13}\left(\begin{array}{c}-49 e^{5} - 44 e^{4} + 14 e^{3} - 78 e^{2} - 65 e + 46
\end{array}\right)
\end{array}$

\begin{tabular}{cccc} 
 Group & Field Extension & Elliptic Curve & $j$ \\
$\left(\dfrac{\ZZ/2\ZZ\oplus}{\ZZ/14\ZZ}\right)$ & $\dfrac{\QQ(\sqrt{-7})[b]}{(b^3 + 5b^2 + 2/7b - 1/49)}$ & $E\left(b,\displaystyle\frac{133b^2 + 749b + 54}{167}\right)$ & -3375 \\
$\left(\dfrac{\ZZ/3\ZZ\oplus}{\ZZ/6\ZZ}\right)$ & $\QQ(\zeta_3,\sqrt[3]{-16})$ & $y^2 = x^3 + 16$ & 0\\
$\left(\dfrac{\ZZ/6\ZZ \oplus}{\ZZ/6\ZZ}\right)$ & $\QQ(\zeta_3,\sqrt[3]{4})$ & $y^2 = x^3 + 1$ & 0\\
$\left(\dfrac{\ZZ/3\ZZ\oplus}{\ZZ/9\ZZ}\right)$ & $\dfrac{\QQ(\zeta_3)[b]}{b^3 - 99b^2 - 90b - 9}$ & $E\left(b,\displaystyle\frac{-2b^2 + 318b - 75}{753}\right)$ & 0
\end{tabular}

\subsection{$K$ is a number field of degree 7.}

$$E(K)[\tors]\in \{0,\ZZ/2\ZZ, \ZZ/3\ZZ, \ZZ/4\ZZ, \ZZ/6\ZZ, \ZZ/2\ZZ \oplus \ZZ/2\ZZ\}.$$

No subgroups occur in degree 7 which do not occur over $\QQ$.

\subsection{$K$ is a number field of degree 8.}

$$E(K)[\tors]\in \begin{cases} \ZZ/m\ZZ & \textrm{for } m = 1,\dots,8,10,12,13,\\ &\hspace{1.5cm} 15,16,20,21,28,30,34,39, \\ \ZZ/2\ZZ \oplus \ZZ/m \ZZ & \textrm{for } m=2,4,6,8,10,12,16,20,\\ \ZZ/4\ZZ \oplus \ZZ/m\ZZ & m = 4,8,12, \\ \ZZ/m\ZZ\oplus \ZZ/m\ZZ & \textrm{for } m = 3,5,6, \\\ZZ/m\ZZ\oplus \ZZ/2m\ZZ & \textrm{for } m = 3,5.  \end{cases}$$

The only subgroups which do not occur over $\QQ$ or a number field of degree dividing 8 are:

$$E(K)[\tors]\in \left\{\begin{array}{c}\ZZ/15\ZZ,\ZZ/16\ZZ, \ZZ/20\ZZ,\ZZ/30\ZZ,\ZZ/34\ZZ,\ZZ/39\ZZ \\ \ZZ/2\ZZ\oplus \ZZ/12\ZZ,\ZZ/2\ZZ\oplus \ZZ/16\ZZ, \ZZ/2\ZZ\oplus \ZZ/20\ZZ \\ \ZZ/6\ZZ\oplus \ZZ/6\ZZ,\ZZ/4\ZZ\oplus \ZZ/8\ZZ,\ZZ/4\ZZ\oplus \ZZ/12\ZZ \\ \ZZ/5\ZZ \oplus \ZZ/5\ZZ,\ZZ/5\ZZ\oplus \ZZ/10\ZZ\end{array}\right\}.$$

We give examples of these and hereon, unless otherwise stated, their fields of definition will be given only by the defining polynomial over $\QQ$.

%

$\begin{array}{l|l}
Group & \ZZ/15\ZZ \\ \hline
j, Field & 0,e^{8} - 3 e^{7} - 2 e^{6} + 9 e^{5} - 6 e^{3} - 2 e^{2} - 3 e + 1 = 0 \\ \hline
b & \frac{1}{61}\left(\begin{array}{c}599 e^{7} - 303 e^{6} - 1758 e^{5} + 786 e^{4}\\ + 1411 e^{3} + 632 e^{2} + 755 e - 307
\end{array}\right) \\ \hline 
c & \frac{1}{61}\left(\begin{array}{c}-8 e^{7} + 68 e^{6} + 8 e^{5} - 238 e^{4} + 28 e^{3} + 260 e^{2} + 50 e - 7
\end{array}\right)
\end{array}$


$\begin{array}{l|l}
Group & \ZZ/16\ZZ \\ \hline
j, Field & 16581375,e^{8} - 2 e^{7} + 6 e^{6} - 9 e^{5} + 10 e^{4} - 8 e^{3} + 6 e^{2} - 3 e + 1 = 0 \\ \hline
b & \frac{1}{2}\left(\begin{array}{c}-16 e^{7} + 131 e^{6} - 234 e^{5} + 268 e^{4} - 227 e^{3} + 175 e^{2} - 90 e + 28
\end{array}\right) \\ \hline 
c & \frac{1}{2}\left(\begin{array}{c}-e^{7} + 7 e^{6} - 15 e^{5} + 16 e^{4} - 14 e^{3} + 10 e^{2} - 6 e + 1
\end{array}\right)
\end{array}$


$\begin{array}{l|l}
Group & \ZZ/20\ZZ \\ \hline
j,Field & 287496,e^{8} - 4 e^{7} + 6 e^{6} - 8 e^{4} + 8 e^{3} - 4 e + 2 = 0 \\ \hline
b & 280 e^{7} - 876 e^{6} + 873 e^{5} + 873 e^{4} - 1553 e^{3} + 710 e^{2} + 762 e - 487 \\ \hline 
c & -25 e^{7} + 82 e^{6} - 88 e^{5} - 71 e^{4} + 154 e^{3} - 77 e^{2} - 66 e + 55
\end{array}$


$\begin{array}{l|l}
Group & \ZZ/30\ZZ \\ \hline
j & j^2 + 191025j - 121287375 =0 \\ \hline 
Field & \begin{array}{c}e^{8} - 3 e^{7} - 2 e^{6} + 9 e^{5} - 6 e^{3} - 2 e^{2} - 3 e + 1 = 0\end{array} \\ \hline
b & \dfrac{1}{61}\left(\begin{array}{c}-1926 e^{7} + 7953 e^{6} - 4967 e^{5} - 12311 e^{4}\\ + 13878 e^{3} - 2797 e^{2} + 7188 e - 1853
\end{array}\right) \\ \hline 
c & \dfrac{1}{61}\left(\begin{array}{c}97 e^{7} - 367 e^{6} + 147 e^{5} + 583 e^{4} - 492 e^{3} + 172 e^{2} - 286 e + 62
\end{array}\right)
\end{array}$


$\begin{array}{l|l}
Group & \ZZ/34\ZZ \\ \hline
j, Field & 1728,e^{8} + 4 e^{7} + 7 e^{6} + 8 e^{5} + 8 e^{4} + 6 e^{3} + 4 e^{2} + 2 e + 1 = 0 \\ \hline
b & \frac{1}{17}\left(\begin{array}{c}-11 e^{7} - 34 e^{6} - 38 e^{5} - 16 e^{4} - 11 e^{3} - 11 e^{2} - e + 3
\end{array}\right) \\ \hline 
c & \frac{1}{17}\left(\begin{array}{c}9 e^{7} + 31 e^{6} + 59 e^{5} + 60 e^{4} + 50 e^{3} + 30 e^{2} + 25 e + 6
\end{array}\right)
\end{array}$


$\begin{array}{l|l}
Group & \ZZ/39\ZZ \\ \hline
j,Field & 0,e^{8} - 2 e^{6} - 3 e^{5} + 3 e^{4} + 3 e^{3} - 2 e^{2} + 1 = 0 \\ \hline
b & -12 e^{7} - 25 e^{6} + 7 e^{5} + 82 e^{4} + 77 e^{3} - 47 e^{2} - 95 e - 35 \\ \hline 
c & -8 e^{7} - 4 e^{6} + 18 e^{5} + 34 e^{4} - 14 e^{3} - 46 e^{2} + 13
\end{array}$


$\begin{array}{l|l}
Group & \ZZ/2\ZZ \oplus\ZZ/12\ZZ \\ \hline
j, Field & 54000,e^{8} - 4 e^{7} + 2 e^{6} + 8 e^{5} - 8 e^{4} + 4 e^{3} - 16 e^{2} + 16 e - 2 = 0 \\ \hline
b & \frac{1}{13}\left(\begin{array}{c}-1083 e^{7} + 2865 e^{6} + 1673 e^{5} - 6205 e^{4}\\ - 5 e^{3} - 4167 e^{2} + 11397 e - 1570
\end{array}\right) \\ \hline 
c & \frac{1}{13}\left(\begin{array}{c}-306 e^{7} + 805 e^{6} + 500 e^{5} - 1801 e^{4}\\ + 27 e^{3} - 1218 e^{2} + 3282 e - 453
\end{array}\right)
\end{array}$


$\begin{array}{l|l}
Group & \ZZ/2\ZZ \oplus\ZZ/16\ZZ \\ \hline
j, Field & -3375,e^{8} + 3 e^{7} + 6 e^{6} + 8 e^{5} + 10 e^{4} + 9 e^{3} + 6 e^{2} + 2 e + 1 = 0 \\ \hline
b & \frac{1}{2}\left(\begin{array}{c}20 e^{7} + 47 e^{6} + 79 e^{5} + 96 e^{4} + 110 e^{3} + 89 e^{2} + 26 e + 18
\end{array}\right) \\ \hline 
c & \frac{1}{2}\left(\begin{array}{c}7 e^{7} + 12 e^{6} + 22 e^{5} + 22 e^{4} + 30 e^{3} + 17 e^{2} + 7 e + 3
\end{array}\right)
\end{array}$


$\begin{array}{l|l}
Group & \ZZ/2\ZZ\oplus\ZZ/20\ZZ \\ \hline
j,Field & 1728,e^{8} - 4 e^{7} + 6 e^{6} - 8 e^{4} + 8 e^{3} - 4 e + 2 = 0 \\ \hline
b & -5 e^{7} + 17 e^{6} - 20 e^{5} - 12 e^{4} + 34 e^{3} - 21 e^{2} - 15 e + 13 \\ \hline 
c & -4 e^{7} + 13 e^{6} - 13 e^{5} - 13 e^{4} + 24 e^{3} - 8 e^{2} - 10 e + 7
\end{array}$


$\begin{array}{l|l}
Group & \ZZ/6\ZZ \oplus\ZZ/6\ZZ \\ \hline
j, Field & -3375,e^{8} + 3 e^{7} + 4 e^{6} + 3 e^{5} + 3 e^{4} + 3 e^{3} + 4 e^{2} + 3 e + 1 = 0 \\ \hline
b & \frac{1}{15}\left(132 e^{7} + 513 e^{6} + 681 e^{5} + 267 e^{4} + 143 e^{3} + 629 e^{2} + 602 e + 193
\right) \\ \hline 
c & \frac{1}{15}\left(\begin{array}{c}-21 e^{7} - 39 e^{6} - 3 e^{5} + 29 e^{4} - 39 e^{3} - 32 e^{2} + 9 e + 11
\end{array}\right)
\end{array}$


$\begin{array}{l|l}
Group & \ZZ/4\ZZ \oplus\ZZ/8\ZZ \\ \hline
j, Field & -3375,e^{8} - e^{6} - 2 e^{5} + e^{4} + 8 e^{3} + 12 e^{2} + 8 e + 2 = 0 \\ \hline
b & \frac{1}{11}\left(\begin{array}{c}103 e^{7} - 85 e^{6} - 43 e^{5} - 153 e^{4} + 222 e^{3} + 655 e^{2} + 663 e + 235
\end{array}\right) \\ \hline 
c & \frac{1}{11}\left(\begin{array}{c}63 e^{7} - 27 e^{6} - 64 e^{5} - 86 e^{4} + 114 e^{3} + 452 e^{2} + 534 e + 206
\end{array}\right)
\end{array}$


$\begin{array}{l|l}
Group & \ZZ/4\ZZ \oplus\ZZ/12\ZZ \\ \hline
j, Field & 0,e^{8} - 2 e^{7} + 2 e^{6} - 2 e^{5} + 7 e^{4} - 10 e^{3} + 8 e^{2} - 4 e + 1 = 0 \\ \hline
b & \frac{1}{11}\left(\begin{array}{c}-32 e^{7} + 66 e^{6} - 53 e^{5} + 57 e^{4} - 220 e^{3} + 320 e^{2} - 188 e + 115
\end{array}\right) \\ \hline 
c & \frac{1}{11}\left(\begin{array}{c}-28 e^{7} + 66 e^{6} - 56 e^{5} + 54 e^{4} - 198 e^{3} + 324 e^{2} - 214 e + 91
\end{array}\right)
\end{array}$


$\begin{array}{l|l} \textrm{Group, Field} & \ZZ/5\ZZ\oplus \ZZ/5\ZZ, \QQ(\zeta_{15}) \\ \hline j,b=c & (0,4\zeta_{15}^7 + 2\zeta_{15}^6 - 2\zeta_{15}^5 -2\zeta_{15}^3 + 4\zeta_{15} + 1)\end{array}$

$\begin{array}{l|l} \textrm{Group, Field} & \ZZ/5\ZZ\oplus \ZZ/10\ZZ, \QQ(\zeta_{20}) \\ \hline j,b=c & (1728,\zeta_4)\end{array}$

\subsection{$K$ is a number field of degree 9.}

$$E(K)[\tors]\in \begin{cases} \ZZ/m\ZZ & \textrm{for } m = 1,2,3,4,6,9,14,18,19,27, \\ \textrm{and }& \ZZ/2\ZZ \oplus \ZZ/2 \ZZ. \end{cases}$$

The subgroups which do not occur over $\QQ$ or a number field of degree 3 are:

$$E(K)[\tors]\in \{ \ZZ/18\ZZ,\ZZ/19\ZZ,\ZZ/27\ZZ\}.$$

Examples of these are:


$\begin{array}{l|l}
Group & \ZZ/18\ZZ \\ \hline
j & 54000 \\ \hline 
Field & \begin{array}{c}e^{9} - 3 e^{8} + 3 e^{7} - 6 e^{6} + 12 e^{5} - 3 e^{4} - 15 e^{3} + 15 e^{2} - 6 e + 1 = 0\end{array} \\ \hline
b & \dfrac{1}{3}\left(\begin{array}{c}217 e^{8} - 235 e^{7} + 202 e^{6} - 904 e^{5} + 841 e^{4}\\ + 971 e^{3} - 1364 e^{2} + 617 e - 113
\end{array}\right) \\ \hline 
c & \dfrac{1}{3}\left(\begin{array}{c}7 e^{8} - 8 e^{7} + 7 e^{6} - 31 e^{5} + 32 e^{4} + 29 e^{3} - 50 e^{2} + 25 e - 5
\end{array}\right)
\end{array}$
%

$\begin{array}{l|l}
Group & \ZZ/19\ZZ \\ \hline
j & -884736 \\ \hline 
Field & \begin{array}{c}e^{9} - e^{8} - 8 e^{7} + 7 e^{6} + 21 e^{5} - 15 e^{4} - 20 e^{3} + 10 e^{2} + 5 e - 1 = 0\end{array} \\ \hline
b & -31 e^{8} + 73 e^{7} + 112 e^{6} - 319 e^{5} - 80 e^{4} + 397 e^{3} - 26 e^{2} - 139 e + 22 \\ \hline 
c & \left(\begin{array}{c}-2 e^{8} + 16 e^{6} - 4 e^{5} - 38 e^{4} + 14 e^{3} + 26 e^{2} - 10 e + 1
\end{array}\right)
\end{array}$


$\begin{array}{l|l}
Group\footnotemark & \ZZ/27\ZZ \\ \hline
j & -12288000   \\ \hline 
Field & \begin{array}{c}e^{9} - 9 e^{7} + 27 e^{5} - 30 e^{3} + 9 e - 1 = 0\end{array} \\ \hline
b & \left(\begin{array}{c}-4282 e^{8} - 507 e^{7} + 38492 e^{6} + 4523 e^{5} - 115156 e^{4}\\ - 13456 e^{3} + 126990 e^{2} + 14789 e - 36852
\end{array}\right) \\ \hline 
c & \left(\begin{array}{c}16 e^{8} + 2 e^{7} - 140 e^{6} - 18 e^{5} + 410 e^{4} + 54 e^{3} - 444 e^{2} - 58 e + 125
\end{array}\right)
\end{array}$\footnotetext{The discriminant of the CM order is -27 and the SPY bounds are sharp here. Typically when the SPY bounds are sharp, $\gcd(\disc(\OO),N) >1$.}

\subsection{$K$ is a number field of degree 10.}

$$E(K)[\tors]\in \begin{cases} \ZZ/m\ZZ & \textrm{for } m = 1,2,3,4,6,7,10,11,22,31,50, \\ \ZZ/2\ZZ \oplus \ZZ/m \ZZ & \textrm{for } m = 2,4,6,22, \\ \textrm{and }& \ZZ/3\ZZ\oplus \ZZ/3\ZZ.\end{cases}$$

The only subgroups which do not occur over $\QQ$ or a number field of degree 2 or 5 are:

$$E(K)[\tors]\in \{\ZZ/22\ZZ,\ZZ/31\ZZ,\ZZ/50\ZZ, \ZZ/2\ZZ\oplus \ZZ/22\ZZ\}.$$

Examples of these are:


$\begin{array}{l|l}
Group & \ZZ/22\ZZ \\ \hline
j & 16581375 \\ \hline 
Field & \begin{array}{c}e^{10} - e^{9} + 2 e^{8} - 4 e^{7} + 5 e^{6} - 3 e^{5} + 3 e^{4} - 6 e^{3} + 7 e^{2} - 4 e + 1 = 0\end{array} \\ \hline
b & \left(\begin{array}{c}-184243 e^{9} + 88117 e^{8} - 299927 e^{7} + 589670 e^{6} - 568675 e^{5}\\ + 223462 e^{4} - 395246 e^{3} + 908033 e^{2} - 759722 e + 279034
\end{array}\right) \\ \hline 
c & \left(\begin{array}{c}1905 e^{9} - 643 e^{8} + 3210 e^{7} - 5564 e^{6} + 5493 e^{5}\\ - 1825 e^{4} + 4191 e^{3} - 8721 e^{2} + 7124 e - 2426
\end{array}\right)
\end{array}$


$\begin{array}{l|l}
Group & \ZZ/31\ZZ \\ \hline
j & 0 \\ \hline 
Field & \begin{array}{c}e^{10} + 2 e^{8} - 3 e^{7} + 3 e^{6} - 7 e^{5} + 8 e^{4} - 7 e^{3} + 7 e^{2} - 4 e + 1 = 0\end{array} \\ \hline
b & \left(\begin{array}{c}43 e^{9} - 33 e^{8} + 45 e^{7} - 222 e^{6} + 133 e^{5}\\ - 335 e^{4} + 508 e^{3} - 326 e^{2} + 368 e - 252
\end{array}\right) \\ \hline 
c & \left(\begin{array}{c}24 e^{9} + 10 e^{8} + 50 e^{7} - 54 e^{6} + 44 e^{5}\\ - 148 e^{4} + 130 e^{3} - 104 e^{2} + 122 e - 43
\end{array}\right)
\end{array}$

%

$\begin{array}{l|l}
Group & \ZZ/50\ZZ \\ \hline
j & 1728 \\ \hline 
Field & \begin{array}{c}e^{10} - 4 e^{9} + 9 e^{8} - 14 e^{7} + 15 e^{6} - 10 e^{5} + 3 e^{4} + 2 e^{3} - 2 e^{2} + 1 = 0\end{array} \\ \hline
b & \left(\begin{array}{c}20 e^{8} - 59 e^{7} + 87 e^{6} - 79 e^{5} + 31 e^{4} + 17 e^{3} - 12 e^{2} + 4 e + 8
\end{array}\right) \\ \hline 
c & \left(\begin{array}{c}e^{9} - 6 e^{8} + 10 e^{7} - 10 e^{6} + 7 e^{5} - e^{4} - e^{3} - 2 e^{2} - 2 e - 1
\end{array}\right)
\end{array}$


$\begin{array}{l|l}
Group & \ZZ/2\ZZ\oplus\ZZ/22\ZZ \\ \hline
j & -3375 \\ \hline 
Field & \begin{array}{c}e^{10} - e^{9} + 2 e^{8} - 4 e^{7} + 5 e^{6} - 3 e^{5} + 3 e^{4} - 6 e^{3} + 7 e^{2} - 4 e + 1 = 0\end{array} \\ \hline
b & \left(\begin{array}{c}-14 e^{9} + 4 e^{8} - 22 e^{7} + 41 e^{6} - 35 e^{5}\\ + 9 e^{4} - 30 e^{3} + 63 e^{2} - 46 e + 10
\end{array}\right) \\ \hline 
c & \left(\begin{array}{c}-6 e^{9} + 5 e^{8} - 9 e^{7} + 23 e^{6} - 22 e^{5} + 10 e^{4} - 13 e^{3} + 33 e^{2} - 32 e + 11
\end{array}\right)
\end{array}$

\subsection{$K$ is a number field of degree 11.}

$$E(K)[\tors]\in \{0,\ZZ/2\ZZ, \ZZ/3\ZZ, \ZZ/4\ZZ, \ZZ/6\ZZ, \ZZ/2\ZZ \oplus \ZZ/2\ZZ\}.$$

No subgroups occur in degree 11 which do not occur over $\QQ$.

\subsection{$K$ is a number field of degree 12.}

$$E(K)[\tors]\in \begin{cases} \ZZ/m\ZZ & \textrm{for } m = 1,\dots,10,12,13,14\\ &\hspace{1.5cm} 18,19,21,26,37,42,57, \\ \ZZ/2\ZZ \oplus \ZZ/m \ZZ & \textrm{for } m=2,4,6,8,10,12,14,18,26,28,42,\\ \ZZ/3\ZZ\oplus \ZZ/m\ZZ & \textrm{for } m = 3,6,9,12,18,21,\\ \ZZ/m\ZZ\oplus \ZZ/m\ZZ &\textrm{for } m= 4,6,7.\end{cases}$$

The only subgroups which do not occur over a number field of degree dividing 12 are:

$$E(K)[\tors] \in \begin{cases} \ZZ/m\ZZ & \textrm{for } m =  28, 37, 42, 57, \\\ZZ/2\ZZ \oplus \ZZ/m \ZZ & \textrm{for } m = 12, 18, 26, 28, 42, \\ \ZZ/3\ZZ\oplus \ZZ/m\ZZ & \textrm{for } m = 12, 18,21, \\ \textrm{ and } &\ZZ/7\ZZ\oplus \ZZ/7\ZZ.\end{cases}$$

These are:

%

$\begin{array}{l|l}
Group & \ZZ/28\ZZ \\ \hline
j & 54000 \\ \hline 
Field & \begin{array}{c}e^{12} - 4 e^{11} + 8 e^{10} - 6 e^{9} - 7 e^{8} + 20 e^{7}\\ - 18 e^{6} - 4 e^{5} + 25 e^{4} - 8 e^{3} - 6 e^{2} + 2 e + 1 = 0\end{array} \\ \hline
b & \dfrac{1}{402}\left(\begin{array}{c}110 e^{11} - 95 e^{10} - 24 e^{9} - 47 e^{8} - 19 e^{7} - 232 e^{6}\\ - 119 e^{5} + 1480 e^{4} + 369 e^{3} - 1017 e^{2} + 149 e + 197
\end{array}\right) \\ \hline 
c & \dfrac{1}{402}\left(\begin{array}{c}116 e^{11} - 423 e^{10} + 853 e^{9} - 655 e^{8} - 768 e^{7} + 2238 e^{6}\\ - 1982 e^{5} - 766 e^{4} + 2806 e^{3} - 277 e^{2} - 441 e - 3\end{array}\right)
\end{array}$

%

$\begin{array}{l|l}
Group & \ZZ/37\ZZ \\ \hline
j & 0 \\ \hline 
Field & \begin{array}{c}e^{12} - 4 e^{11} + 11 e^{10} - 21 e^{9} + 32 e^{8} - 40 e^{7}\\ + 45 e^{6} - 46 e^{5} + 40 e^{4} - 26 e^{3} + 12 e^{2} - 4 e + 1 = 0\end{array} \\ \hline
b & \dfrac{1}{37}\left(\begin{array}{c}-196 e^{11} + 657 e^{10} - 1789 e^{9} + 3384 e^{8} - 5292 e^{7} + 6890 e^{6}\\ - 7695 e^{5} + 7154 e^{4} - 4851 e^{3} + 2221 e^{2} - 773 e + 181
\end{array}\right) \\ \hline 
c & \dfrac{1}{37}\left(\begin{array}{c}24 e^{11} - 162 e^{10} + 432 e^{9} - 952 e^{8} + 1462 e^{7} - 1928 e^{6}\\ + 2090 e^{5} - 2060 e^{4} + 1630 e^{3} - 870 e^{2} + 294 e - 109
\end{array}\right)
\end{array}$

%

$\begin{array}{l|l}
Group & \ZZ/42\ZZ\footnotemark \\ \hline
j & 54000 \\ \hline 
Field & \begin{array}{c}e^{12} - 4 e^{11} + 8 e^{10} - 11 e^{9} + 13 e^{8} - 14 e^{7}\\ + 15 e^{6} - 14 e^{5} + 7 e^{4} + 3 e^{3} - 5 e^{2} + e + 1 = 0\end{array} \\ \hline
b & \left(\begin{array}{c}1416 e^{11} - 6140 e^{10} + 13362 e^{9} - 19953 e^{8} + 24872 e^{7} - 27802 e^{6}\\ + 30076 e^{5} - 29333 e^{4} + 19087 e^{3} - 1483 e^{2} - 7097 e + 4041
\end{array}\right) \\ \hline 
c & \left(\begin{array}{c}27 e^{11} - 138 e^{10} + 342 e^{9} - 563 e^{8} + 754 e^{7} - 896 e^{6}\\ + 1000 e^{5} - 1038 e^{4} + 847 e^{3} - 370 e^{2} - 7 e + 66
\end{array}\right)
\end{array}$\footnotetext{Note that this elliptic curve is isogenous to the one with $j = 0$ and torsion $\ZZ/2\ZZ\oplus\ZZ/42\ZZ$.}

%

$\begin{array}{l|l}
Group & \ZZ/57\ZZ \\ \hline
j & 0 \\ \hline 
Field & \begin{array}{c}e^{12} - 2 e^{11} + 5 e^{10} - 10 e^{9} + 16 e^{8} - 22 e^{7}\\ + 30 e^{6} - 31 e^{5} + 28 e^{4} - 27 e^{3} + 19 e^{2} - 7 e + 1 = 0\end{array} \\ \hline
b & \left(\begin{array}{c}18509 e^{11} - 25122 e^{10} + 76123 e^{9} - 135966 e^{8}\\ + 207749 e^{7} - 272291 e^{6} + 378051 e^{5} - 328154 e^{4}\\ + 303397 e^{3} - 302371 e^{2} + 154326 e - 27788
\end{array}\right) \\ \hline 
c & \left(\begin{array}{c}128 e^{11} - 62 e^{10} + 446 e^{9} - 532 e^{8} + 876 e^{7} - 986 e^{6}\\ + 1542 e^{5} - 670 e^{4} + 1136 e^{3} - 872 e^{2} + 18 e + 71
\end{array}\right)
\end{array}$


$\begin{array}{l|l}
Group & \ZZ/2\ZZ\oplus\ZZ/12\ZZ \\ \hline
j & j^3 + 3491750j^2 - 5151296875j + 12771880859375 =0 \\ \hline 
Field & \begin{array}{c}e^{12} - 4e^{11} + 11e^{10} - 28e^9 + 63e^8 - 114e^7\\ + 161e^6 - 174e^5 + 141e^4 - 82e^3 + 33e^2 - 8e + 1 = 0\end{array} \\ \hline
b & \dfrac{1}{8}\left(\begin{array}{c}46e^{11} - 165e^{10} + 327e^9 - 914e^8 + 1949e^7 - 2883e^6\\ + 3279e^5 - 2583e^4 + 1240e^3 - 576e^2 + 169e - 34\end{array}\right) \\ \hline 
c & \dfrac{1}{8}\left(\begin{array}{c}86e^{11} - 295e^{10} + 760e^9 - 1947e^8 + 4217e^7 - 7168e^6\\ + 9344e^5 - 9004e^4 + 6265e^3 - 2894e^2 + 769e - 103\end{array}\right)
\end{array}$


$\begin{array}{l|l}
Group & \ZZ/2\ZZ\oplus\ZZ/18\ZZ \\ \hline
j & 8000 \\ \hline 
Field & \begin{array}{c}e^{12} - 4e^{11} + 4e^{10} + 8e^9 - 25e^8 +  24e^7\\ + 4e^6 - 36e^5 + 46e^4 - 32e^3 + 14e^2 - 4e + 1 = 0\end{array} \\ \hline
b & \dfrac{1}{369}\left(\begin{array}{c}-1450e^{11} + 3898e^{10} + 304e^9 - 13733e^8 + 17660e^7 - 2419e^6\\ -  19740e^5 + 26634e^4 - 18553e^3 + 5681e^2 - 2148e + 155\end{array}\right) \\ \hline 
c & \dfrac{1}{123}\left(\begin{array}{c}1108e^{11} - 3354e^{10} + 756e^9 + 10638e^8 - 17099e^7 + 6109e^6\\ + 14805e^5 - 25449e^4 + 20608e^3 - 8740e^2 + 2841e - 680\end{array}\right)
\end{array}$

%

$\begin{array}{l|l}
Group & \ZZ/2\ZZ\oplus\ZZ/26\ZZ \\ \hline
j & 0 \\ \hline 
Field & \begin{array}{c}e^{12} - 5e^{11} + 11e^{10} - 9e^9 - 7e^8 + 24e^7\\ - 21e^6 + 21e^4 - 25e^3 + 16e^2 - 6e + 1 = 0\end{array} \\ \hline
b & \dfrac{1}{91}\left(\begin{array}{c}-781 e^{11} + 3351 e^{10} - 5753 e^{9} + 969 e^{8} + 9554 e^{7} - 12544 e^{6}\\ + 1862 e^{5} + 8736 e^{4} - 11312 e^{3} + 6330 e^{2} - 1292 e - 19
\end{array}\right) \\ \hline 
c & \dfrac{1}{91}\left(\begin{array}{c}-198 e^{11} + 1336 e^{10} - 3552 e^{9} + 3848 e^{8} + 1816 e^{7} - 9198 e^{6}\\ + 8344 e^{5} + 1134 e^{4} - 8428 e^{3} + 8618 e^{2} - 4202 e + 785
\end{array}\right)
\end{array}$

%

$\begin{array}{l|l}
Group & \ZZ/2\ZZ \oplus \ZZ/28\ZZ \\ \hline
j & -3375 \\ \hline 
Field & \begin{array}{c}e^{12} - 4 e^{11} + 5 e^{10} + 3 e^{9} - 11 e^{8} - 3 e^{7}\\ + 35 e^{6} - 47 e^{5} + 27 e^{4} - 4 e^{3} - e^{2} - e + 1 = 0\end{array} \\ \hline
b & \dfrac{1}{43}\left(\begin{array}{c}1252 e^{11} - 3557 e^{10} + 1151 e^{9} + 6917 e^{8} - 4746 e^{7} - 13843 e^{6}\\ + 26751 e^{5} - 17575 e^{4} + 2938 e^{3} + 964 e^{2} + 523 e - 674
\end{array}\right) \\ \hline 
c & \dfrac{1}{43}\left(\begin{array}{c}103 e^{11} - 501 e^{10} + 525 e^{9} + 659 e^{8} - 1439 e^{7} - 1039 e^{6}\\ + 4489 e^{5} - 4371 e^{4} + 1157 e^{3} + 216 e^{2} - 45 e - 219
\end{array}\right)
\end{array}$


$\begin{array}{l|l}
Group & \ZZ/2\ZZ\oplus \ZZ/42\ZZ \\ \hline
j & 0 \\ \hline 
Field & \begin{array}{c}e^{12} - 4 e^{11} + 8 e^{10} - 11 e^{9} + 13 e^{8} - 14 e^{7}\\ + 15 e^{6} - 14 e^{5} + 7 e^{4} + 3 e^{3} - 5 e^{2} + e + 1 = 0\end{array} \\ \hline
b & \left(\begin{array}{c}17 e^{11} - 40 e^{10} + 61 e^{9} - 71 e^{8} + 82 e^{7} - 81 e^{6}\\ + 95 e^{5} - 59 e^{4} - 12 e^{3} + 41 e^{2} - 6 e - 7
\end{array}\right) \\ \hline 
c & \left(\begin{array}{c}6 e^{11} - 18 e^{10} + 28 e^{9} - 34 e^{8} + 38 e^{7} - 40 e^{6}\\ + 44 e^{5} - 36 e^{4} + 20 e^{2} - 8 e - 5
\end{array}\right)
\end{array}$


$\begin{array}{l|l}
Group & \ZZ/3\ZZ\oplus\ZZ/12\ZZ \\ \hline
j & 54000 \\ \hline 
Field & \begin{array}{c}e^{12} - 3e^{10} - 2e^9 + 12e^8 - 6e^7 - 3e^6\\ - 12e^5 + 36e^4 - 38e^3 + 21e^2 - 6e + 1 = 0\end{array} \\ \hline
b & \dfrac{1}{7843}\left(\begin{array}{c}-208660e^{11} - 118345e^{10} + 606681e^9 + 844271e^8\\ -    2044930e^7 - 126154e^6 + 712244e^5 + 3124663e^4\\ -
    5460349e^3 + 4188689e^2 - 1329972e + 272478\end{array}\right) \\ \hline 
c & \dfrac{1}{7843}\left(\begin{array}{c}55805e^{11} + 25237e^{10} - 174599e^9 - 227269e^8\\ +
    578726e^7 + 22958e^6 - 227941e^5 - 894570e^4\\ +
    1575065e^3 - 1231382e^2 + 392162e - 82582\end{array}\right)
\end{array}$

$\begin{array}{l|l}
Group & \ZZ/3\ZZ\oplus\ZZ/18\ZZ \\ \hline
j & 8000 \\ \hline 
Field & \begin{array}{c}e^{12} - 4 e^{11} + 12 e^{10} - 24 e^{9} + 38 e^{8} - 50 e^{7}\\ + 52 e^{6} - 48 e^{5} + 36 e^{4} - 24 e^{3} + 15 e^{2} - 6 e + 3 = 0\end{array} \\ \hline
b & \dfrac{1}{53}\left(\begin{array}{c}-262 e^{11} + 1376 e^{10} - 4458 e^{9} + 10137 e^{8} - 17241 e^{7} + 23682 e^{6}\\ - 25526 e^{5} + 22112 e^{4} - 14766 e^{3} + 6941 e^{2} - 2382 e + 151
\end{array}\right) \\ \hline 
c & \dfrac{1}{53}\left(\begin{array}{c}131 e^{11} - 476 e^{10} + 1328 e^{9} - 2339 e^{8} + 3188 e^{7} - 3414 e^{6}\\ + 2481 e^{5} - 1304 e^{4} + 69 e^{3} + 319 e^{2} - 293 e + 110
\end{array}\right)
\end{array}$

$\begin{array}{l|l}
Group & \ZZ/3\ZZ\oplus\ZZ/21\ZZ \\ \hline
j & 0 \\ \hline 
Field & \begin{array}{c}e^{12} - 6 e^{11} + 18 e^{10} - 35 e^{9} + 54 e^{8} - 72 e^{7}\\ + 84 e^{6} - 81 e^{5} + 66 e^{4} - 44 e^{3} + 21 e^{2} - 6 e + 1 = 0\end{array} \\ \hline
b & \dfrac{1}{49}\left(\begin{array}{c}-4894 e^{11} + 30046 e^{10} - 87461 e^{9} + 154268 e^{8}\\ - 201926 e^{7} + 235109 e^{6} - 256228 e^{5} + 214939 e^{4}\\ - 117237 e^{3} + 38057 e^{2} - 7246 e + 352
\end{array}\right) \\ \hline 
c & \dfrac{1}{49}\left(\begin{array}{c}290 e^{11} - 1742 e^{10} + 4986 e^{9} - 8686 e^{8} + 11366 e^{7} - 13286 e^{6}\\ + 14434 e^{5} - 12010 e^{4} + 6618 e^{3} - 2248 e^{2} + 458 e - 41
\end{array}\right)
\end{array}$

$\begin{array}{l|l} \textrm{Group, Field} & \ZZ/7\ZZ\oplus \ZZ/7\ZZ, \QQ(\zeta_{21}) \\ \hline (j,b,c) & (0,\zeta_3, -1)\end{array}$

\subsection{$K$ is a number field of degree 13.}

$$E(K)[\tors]\in \{0,\ZZ/2\ZZ, \ZZ/3\ZZ, \ZZ/4\ZZ, \ZZ/6\ZZ, \ZZ/2\ZZ \oplus \ZZ/2\ZZ\}.$$

No subgroups occur in degree 13 which do not occur over $\QQ$.

\end{document}